\documentclass[a4paper,reqno,10pt]{amsart}

\usepackage{epsf,graphicx,epsfig}
\usepackage{amscd}
\usepackage{amsmath,latexsym,amssymb,amsthm}
\usepackage[nospace,noadjust]{cite}
\usepackage{textcomp}
\usepackage{setspace,cite}
\usepackage{mathrsfs,amsmath}
\usepackage{lscape,fancyhdr,fancybox}
\usepackage{stmaryrd}
\usepackage[all,cmtip]{xy}
\usepackage{tikz}
\usepackage{cancel}
\usetikzlibrary{shapes,arrows,decorations.markings}
\usepackage{graphicx}

\usepackage{color}
\usepackage{bbm}
\usepackage{xcolor}
\usepackage{url}
\usepackage{enumerate}
\usepackage[mathscr]{euscript}
  \theoremstyle{plain}

\swapnumbers
    \newtheorem{thm}{Theorem}[section]
    \newtheorem{prop}[thm]{Proposition}

    \newtheorem{subsec}[thm]{}
\theoremstyle{definition}
    \newtheorem{defn}[thm]{Definition}
        \newtheorem{remark}[thm]{Remark}
    \newtheorem{exam}[thm]{Example}

\theoremstyle{remark}

\title{}
\author{}
\date{}
\usepackage{amssymb}

\usepackage{hyperref}
\hypersetup{
colorlinks,
citecolor=blue,
filecolor=black,
linkcolor=blue,
urlcolor=black
}

\begin{document}

\title[Matching relative Rota-Baxter algebras and matching dendriform algebras]{Matching relative Rota-Baxter algebras, matching dendriform algebras and their cohomologies}

\author{Ramkrishna Mandal}
\address{Department of Mathematics,
Indian Institute of Technology, Kharagpur 721302, West Bengal, India.}
\email{ramkrishnamandal430@gmail.com}

\author{Apurba Das}
\address{Department of Mathematics,
Indian Institute of Technology, Kharagpur 721302, West Bengal, India.}
\email{apurbadas348@gmail.com, apurbadas348@maths.iitkgp.ac.in}

\begin{abstract}
The notion of matching Rota-Baxter algebras was recently introduced by Gao, Guo and Zhang [{\em J. Algebra} 552 (2020) 134-170] motivated by the study of algebraic renormalization of regularity structures. The concept of matching Rota-Baxter algebras generalizes multiple integral operators with kernels. The same authors also introduced matching dendriform algebras as the underlying structure of matching Rota-Baxter algebras. In this paper, we introduce matching relative Rota-Baxter algebras that are also related to matching dendriform algebras. We define a matching associative Yang-Baxter equation whose solutions give rise to matching relative Rota-Baxter algebras. Next, we introduce the cohomology of a matching relative Rota-Baxter algebra as a byproduct of the classical Hochschild cohomology and a new cohomology induced by the matching operators. As an application, we show that our cohomology governs the formal deformation theory of the matching relative Rota-Baxter algebra. Finally, using multiplicative nonsymmetric operads, we define the cohomology of a matching dendriform algebra and show that there is a morphism from the cohomology of a matching relative Rota-Baxter algebra to the cohomology of the induced matching dendriform algebra. We end this paper by considering homotopy matching dendriform algebras.
\end{abstract}

\maketitle

%\curraddr{}
%\email{}

%\subjclass[2010]{}
%\keywords{}

\medskip

\begin{center}

\noindent {2020 MSC classifications:} {17B38,16S80,17B70.}

\noindent {Keywords:} Matching Rota-Baxter algebras, matching dendriform algebras, Cohomology, Deformations.

\end{center}

%\medskip

%\noindent {\sf Date of resubmission:} July 26, 2021.

\thispagestyle{empty}

\tableofcontents

%\vspace{0.2cm}

\medskip

%\textcolor{red}{When a pre-Leibniz algebra is called commutative (zinbiel version)? When symmetric (symmetric Leibniz version)?}

\section{Introduction}\label{sec1}
Rota-Baxter operators first appeared in the work of G. Baxter in the fluctuation theory of probability \cite{baxter}. Subsequently, such operators were studied by G. C. Rota \cite{rota}, P. Cartier \cite{cartier}, and F. V. Atkinson \cite{atkinson} among others. Rota-Baxter operators are the algebraic abstraction of the integral operator. In the last twenty years, Rota-Baxter operators experienced several developments including the connection with the algebraic approach to renormalization in perturbative quantum field theory \cite{connes}. Additionally, Rota-Baxter operators are closely related to the splitting of algebras (in particular, related to dendriform algebras and pre-Lie algebras) \cite{aguiar,bai-olivia}, the associative analogue of the classical Yang-Baxter equation \cite{aguiar2,aguiar4},  infinitesimal bialgebras \cite{aguiar3}, mixable shuffle product algebras \cite{guo-keigher} and many others. A Rota-Baxter algebra is an associative algebra $A$ equipped with a Rota-Baxter operator. See \cite{guo-book,ebrahimi,das-rota,das-mishra,tang,laza} for some developments of Rota-Baxter algebras.\\

\subsection{Relative Rota-Baxter algebras and dendriform algebras}
The notion of the relative Rota-Baxter operator (also known as the generalized Rota-Baxter operator or $\mathcal{O}$-operator in the literature) was first introduced by K. Uchino \cite{uchino} as a generalization of the Rota-Baxter operator in the presence of a bimodule. Let $A$ be an associative algebra and $M$ be an $A$-bimodule. Denote the associative multiplication on $A$ simply by dot $~\cdot~$ and both the left and right $A$-actions on $M$ simply by $~\cdot_M~$ if no confusion arises. A linear map $P : M \rightarrow A $ is a relative Rota-Baxter operator on $M$ over the algebra $A$ if the linear map $P$ satisfies
\begin{align*}
P(u) \cdot P(v) = P\big( P(u)\cdot_M v ~+~ u\cdot_M P(v) \big),  \text{ for } u,v\in M.    
\end{align*}  
In the same article, Uchino observed that relative Rota-Baxter operators are closely related to dendriform algebras, an algebraic structure introduced by J.-L. Loday \cite{loday}. A triple $(A, M, P)$ consisting of an algebra $A$, an $A$-bimodule $M$ and a relative Rota-Baxter operator $P$ is called a relative Rota-Baxter algebra. In \cite{das-rota,das-mishra}, the authors introduced the cohomology of relative Rota-Baxter operators and relative Rota-Baxter algebras and study their formal deformations. They also find the connection between the cohomology of a relative Rota-Baxter algebra and the cohomology of the induced dendriform algebra. See also \cite{das-dend,gubarev} for some other developments of relative Rota-Baxter algebras and dendriform algebras.\\

\subsection{Matching Rota-Baxter algebras and matching dendriform algebras}
Algebraic structures are often defined by multiple operations which has various compatibility conditions. Well-known structures that are defined by multiple operations include dendriform algebras \cite{loday}, diassociative algebras \cite{loday}, Poisson algebras \cite{sch} and compatible Lie algebras \cite{magri}. Recently, L. Foissy \cite{foissy} introduced a notion of multiple pre-Lie algebra motivated by the classical work of algebraic renormalization of regularity structures \cite{bruned}. Note that a multiple pre-Lie algebra is defined by multiple operations labelled by a nonempty set $X$. To better understand multiple pre-Lie algebras (also called matching pre-Lie algebras) from the point of view of Rota-Baxter algebras, X. Guo, L. Guo and Y. Zhang introduced a notion of matching Rota-Baxter algebras \cite{gao-guo-zhang}. Let $X$ be a nonempty set. A matching Rota-Baxter algebra is a pair $(A,\{P_x\}_{x\in X})$ consisting of an algebra $A$ and a collection $\{P_x\}_{x\in X}$ of linear maps on $A$ labelled by the set $X$ that satisfies
\begin{align}\label{rrba-iden}
 P_x(a)\cdot P_y(b) = P_x(a\cdot P_y(b)) + P_y(P_x(a)\cdot b),~~ \mathrm{for}~ a,b \in A~~ \mathrm{and}~ x,y \in X.  \end{align}
 Like Rota-Baxter algebras generalize integral operators, the notion of matching Rota-Baxter algebras
 generalize multiple integral operators with kernels. More precisely, let $A= \mathrm{Cont}(\mathbb{R})$ be the algebra of continuous functions on the real line. Consider a family of functions (called kernels in integral equations) $\{ k_\omega \}_{\omega \in X}$ in $A$ parametrized by the set $X$. We define a collection $\{ I_\omega : A \rightarrow A \}_{\omega \in X}$ of linear maps  by
 \begin{align*}
     \big( I_\omega (f) \big) (x) = \int_{0}^x k_\omega (t) f(t) dt, \text{ for } \omega \in X,  f \in A  \text{ and } x \in \mathbb{R}.
 \end{align*}
 Then $(A, \{ I_\omega \}_{\omega \in X})$ is a matching Rota-Baxter algebra. In the paper \cite{gao-guo-zhang}, the authors also introduced a notion of matching dendriform algebra as the structure behind a matching Rota-Baxter algebra, and showed that a matching dendriform algebra gives rise to a matching pre-Lie algebra in the sense of Foissy. The same set of authors also constructs the free commutative matching Rota-Baxter algebra using the shuffle product with multiple decorations and showed that such free algebra gives rise to the free matching Zinbiel algebra (commutative version of the free matching dendriform algebra) \cite{gao-guo-zhang2}. These are the only results available in the literature about matching Rota-Baxter algebras and matching dendriform algebras.\\

\subsection{Matching relative Rota-Baxter algebras and matching dendriform algebras (and their cohomological study)}
In this paper, we first introduce matching relative Rota-Baxter algebras as a generalization of matching  Rota-Baxter algebras. A matching relative Rota-Baxter algebra is a triple $(A,M,\{P_x\}_{x\in X})$ consisting of an algebra $A$, an $A$-bimodule $M$ and a collection $\{P_x:M\rightarrow A\}_{x\in X}$ of linear maps satisfying the relative version of (\ref{rrba-iden}). We introduce a notion of `matching associative Yang-Baxter equation' associated with an algebra $A$, whose solutions are called matching associative ${r}$-matrices. We show that matching associative $r$-matrices give rise to matching relative Rota-Baxter algebras. Next, we show that a matching relative Rota-Baxter algebra induces a matching dendriform algebra structure. Conversely, a matching dendriform algebra gives rise to a matching relative Rota-Baxter algebra so that the induced matching dendriform structure coincides with the given one.\\

Our main aim in this paper is to study matching (relative) Rota-Baxter algebras and matching dendriform algebras from cohomological points of view. Given an associative algebra $A$ and an $A$-bimodule $M$, we first construct a graded Lie algebra whose Maurer-Cartan elements are given by a collection of linear maps of the form $\{P_x\}_{x\in X}$ that satisfies the relative version of (\ref{rrba-iden}). Thus, a matching relative Rota-Baxter algebra $(A,M,\{P_x\}_{x\in X})$ gives rise to a Maurer-Cartan element and hence induces a cohomology. This cohomology is mainly associated with the collection of operators $\{P_x\}_{x\in X}$. As a byproduct of the above cohomology with the classical Hochschild cohomology of the algebra $A$, we then introduce the cohomology of the matching relative Rota-Baxter algebra $(A,M,\{P_x\}_{x\in X})$. We show that this cohomology governs the formal one-parameter deformation theory of the matching relative Rota-Baxter algebra $(A,M,\{P_x\}_{x\in X})$.  We also consider infinitesimal deformations and classify equivalence classes of infinitesimal deformations in terms of the second cohomology group of the matching relative Rota-Baxter algebra $(A,M,\{P_x\}_{x\in X})$.

\medskip

Next, we study the cohomology of matching dendriform algebras. Given a $\mathbb{K}$-module $D$, we first construct a nonsymmetric operad $\mathcal{O}_D$ whose multiplications correspond to matching dendriform algebra structures on $D$. Using this characterization, we define the cohomology of a matching dendriform algebra. Finally, we show that there is a morphism from the cohomology of a matching relative Rota-Baxter algebra to the cohomology of the induced matching dendriform algebra. Finally, we consider the homotopy version of matching relative Rota-Baxter algebras and matching dendriform algebras. We end this paper by generalizing some relations between matching relative Rota-Baxter algebras and matching dendriform algebras in the homotopy context.

\medskip

\subsection{Organization of the paper.}
In section \ref{sec2}, we recall some preliminaries on the cohomology of relative Rota-Baxter algebras. In section \ref{sec3}, we introduce matching relative Rota-Baxter algebras and the matching associative Yang-Baxter equation and find various relations with matching dendriform algebras. In section \ref{sec4} and \ref{sec5}, we respectively study cohomology and deformations of matching relative Rota-Baxter algebras. Finally, in section \ref{sec6}, we define the cohomology of a matching dendriform algebra and find a relation with the cohomology of matching relative Rota-Baxter algebras. In the same section, we also discuss homotopy matching dendriform algebras.

\medskip

Throughout the paper, $\mathbb{K}$ is a commutative unital ring with char $\mathbb{K}\neq{2}$. All modules, linear maps, $\mathrm{Hom}$ spaces and tensor products are over $\mathbb{K}$ unless specified otherwise. We also assume throughout that $X$ is a nonempty set whose elements are usually denoted by $x,y,z,\ldots$.

\medskip

\section{Background on relative Rota-Baxter algebras}\label{sec2}
In this section, we fix some notations and recall the cohomology of relative Rota-Baxter algebras. Our main references are \cite{das-rota,das-mishra}.

 Let $(A,\mu)$ be an associative algebra. That is, $A$ is a $\mathbb{K}$-module equipped with a $\mathbb{K}$-linear map (called the multiplication map) $\mu:A\otimes A \rightarrow A$,  $a \otimes b \mapsto \mu(a,b)=a\cdot b~$  satisfying the usual associativity $(a\cdot b)\cdot c = a\cdot(b\cdot c)$, for all $a,b,c \in A$. Throughout the paper, we denote an associative algebra $(A,\mu)$ simply by A, when the multiplication map $\mu$ is clear from the context.
 
 \medskip
 
 Let $A$ be an associative algebra. An $A$-bimodule consists of a $\mathbb{K}$-module $M$ together with two linear maps (called the left and right $A$-action, respectively) $l_{M}: A\otimes M \rightarrow M,~~ a\otimes u\mapsto l_M(a,u)= a\cdot_M u$ and $r_{M}: M\otimes A \rightarrow M,~~ u\otimes a\mapsto r_M(u, a)= u\cdot_M a$  satisfying
\begin{align*}
(a \cdot b) \cdot_M u = a\cdot_M(b\cdot_Mu), \quad  (a\cdot_M u)\cdot_M b = a\cdot_M(u\cdot_M b)   \quad \text{and} \quad  (u\cdot_Ma)\cdot_M b = u\cdot_M(a\cdot b),
\end{align*}
for all $a,b\in A$ and $u\in M$. We denote an $A$-bimodule as above by $(M,l_M,r_M)$ or simply by $M$ when the left and right $A$-actions are clear from the context. Note that any associative algebra $A$ can be regarded as an $A$-bimodule in which both the left and right $A$-actions are given by the algebra multiplication map. This is called the adjoint $A$-bimodule.

\medskip

Let $A$ be an associative algebra and $M$ be an $A$-bimodule. The Hochschild cochain complex of $A$ with coefficients in $M$ is given by $\{C^{\ast}_{\mathrm{Hoch}}(A,M),\delta_{\mathrm{Hoch}}\}$, where $C^{n}_{\mathrm{Hoch}}(A,M) := \mathrm{Hom}(A^{\otimes n},M)$ for $n\geq 0$ , and the coboundary map $\delta_{\mathrm{Hoch}} : C^n_{\mathrm{Hoch}}(A,M)\rightarrow C^{n+1}_{\mathrm{Hoch}}(A,M) $ is given by
\begin{align*}
 \delta_{\mathrm{Hoch}}(f)(a_1,a_2,\ldots,a_{n+1})  =~& a_1\cdot_M f(a_2,\ldots,a_{n+1}) + \sum_{i=1}^{n}(-1)^i f(a_1,\ldots, a_i\cdot a_{i+1},\ldots ,a_{n+1})\\
 & + (-1)^{n+1} f(a_1,\ldots,a_n)\cdot_M a_{n+1},
 \end{align*}
for $f\in C^n_{\mathrm{Hoch}}(A,M)$ and $a_1,a_2,\ldots,a_{n+1}\in A$. The corresponding cohomology groups are called the Hochschild cohomology groups of $A$ with coefficients in $M$. They are denoted by $H^\bullet_\mathrm{Hoch} (A, M)$.

\medskip

Given an associative algebra $A$ and an $A$-bimodule $M$, a linear map $P:M\rightarrow A$ is called a {\bf relative Rota-Baxter operator} on $A$ with respect to the bimodule $M$ if the linear map $P$ satisfies
\begin{align*}
 P(u)\cdot P(v) = P\big(P(u)\cdot_M v ~+~ u\cdot _MP(v)\big), \text{ for } u,v\in M.  
\end{align*}
\begin{defn}
A {\bf relative Rota-Baxter algebra} is a triple $(A,M,P)$ consisting of an associative algebra $A$, an $A$-bimodule $M$ and a relative Rota-Baxter operator $P:M\rightarrow A$.
\end{defn}

Let $(A,M,P)$ be a relative Rota-Baxter algebra. Then $M$ inherits an associative algebra structure with the multiplication given by
\begin{align*}
u \star v = P(u)\cdot_M v ~+~ u\cdot_MP(v),   \text{ for } u,v\in M.
\end{align*}
We denote this associative algebra simply by $M_P$. Moreover, there is an $M_P$-bimodule structure on $A$ with the left and right $M_P$-actions are given by
\begin{align*}
\overline{l} : M_P \otimes A\rightarrow A,~ ~  \overline{l}(u,a):= P(u)\cdot a - P(u\cdot_M a),\\
\overline{r}: A\otimes M_P\rightarrow A ,~~ \overline{r}(a,u):= a\cdot P(u)-P(a\cdot_ Mu),
\end{align*}
for $u\in M_P$ and $a\in A$. In \cite{das-rota} the author considered the Hochschild cohomology of the associative algebra $M_P$ with coefficients in the above $M_P$-bimodule $(A,\overline{l},\overline{r})$. It has been observed that such cohomology governs the formal deformation theory of the relative Rota-Baxter operator $P$. However, this cohomology is not suffice to study simultaneous deformations of the underlying algebra $A$, the $A$-bimodule $M$ and the relative Rota-Baxter operator $P$. To study such simultaneous deformations, the authors in \cite{das-mishra} considered a new cohomology theory associated with a relative Rota-Baxter algebra.

\medskip

Let $(A,M,P)$ be a relative Rota-Baxter algebra. Then for each $n \geq 0$, the $n$-th cochain group $C^n_{\mathrm{rRBA}}(A,M,P)$ of the relative Rota-Baxter algebra $(A,M,P)$ is given by
\begin{align*}
    C^n_{\mathrm{rRBA}}(A,M,P)=
\begin{cases}
    0,& \text{if } n= 0\\
    \mathrm{Hom}(A,A)\oplus \mathrm{Hom}(M,M),              & \text{if } n = 1\\
    \mathrm{Hom}(A^{\otimes n},A) \oplus \mathrm{Hom}(\mathcal{A}^{n-1,1},M) \oplus \mathrm{Hom}(M^{\otimes{n-1}},A),              & \text{if } n \geq 2.
\end{cases}
\end{align*}
Here $\mathcal{A}^{n-1,1}$ is the direct sum of all possible $n$ tensor powers of $A$ and $M$ in which $A$ appears $(n-1)$ times and $M$ appears exactly once. Then there is a map $\delta_{\mathrm{rRBA}} : C^n_{\mathrm{rRBA}}(A,M,P)\rightarrow C^{n+1}_{\mathrm{rRBA}}(A,M,P)$ given by
\begin{align*}
 \delta_{\mathrm{rRBA}}(\alpha,\beta,\gamma) = \big(\delta_{\mathrm{Hoch}}(\alpha),~~ \delta^\alpha_{\mathrm{Hoch}}(\beta),~~ \delta_{\mathrm{Hoch}}^{M_P, A} (\gamma) + h_P(\alpha,\beta)\big), \text{ for } (\alpha,\beta,\gamma)\in C^n_{\mathrm{rRBA}}(A,M,P).
\end{align*}
For any $\alpha \in \mathrm{Hom}(A^{\otimes n},A)$, the map $\delta^\alpha_{\mathrm{Hoch}} : \mathrm{Hom}(\mathcal{A}^{n-1,1},M)\rightarrow \mathrm{Hom}(\mathcal{A}^{n,1},M)$ in the above definition is given by
\begin{align*}
\delta^\alpha_{\mathrm{Hoch}}(\beta)(a_1,a_2,\ldots,a_{n+1}) =~& (l_M+r_M)\big(a_1,(\alpha+\beta)(a_2,\ldots,a_{n+1})\big)  \\
&+ \sum_{i=1}^{n} (-1)^i ~\beta \big( a_1,\ldots,(\mu+l_M+r_M)(a_i,a_{i+1}),\ldots,a_{n+1}\big) & \\
&+ (-1)^{n+1} (l_M+r_M)\big((\alpha+\beta)(a_1,\ldots,a_n),a_{n+1}\big),
\end{align*}
for $\beta\in \mathrm{Hom}(\mathcal{A}^{n-1,1},M)~~\mathrm{and}~~ (a_1 , a_2,\ldots, a_{n+1})~~\in \mathcal{A}^{n,1}$. The map $\delta_\mathrm{Hoch}^{M_P, A}$ is the Hochschild coboundary map of the associative algebra $M_P$ with coefficients in the $M_P$-bimodule $A$.
Finally, the map  
\begin{align*}
h_P:\mathrm{Hom}(A^{\otimes n},A)\oplus \mathrm{Hom}(\mathcal{A}^{n-1,1},M) \rightarrow  \mathrm{Hom}(M^{\otimes n},A) \text{ is given by}
\end{align*}
\begin{align*}
h_P (\alpha,\beta)(u_1,u_2,\ldots,u_n) = (-1)^n \big\{\alpha \big(  P(u_1),\ldots, P(u_n) \big) - \sum_{i=1}^{n} P\circ \beta \big( P(u_1),\ldots,u_i,\ldots,P(u_n) \big) \big\},
\end{align*}
for $u_1,\ldots,u_n \in M$. It has been shown that $(\delta_{\mathrm{rRBA}})^2=0$. In other words, $\{C^{\ast}_{\mathrm{rRBA}}(A,M,P), \delta_{\mathrm{rRBA}}\}$ is a cochain complex. The corresponding cohomology groups are called the cohomology of the relative Rota-Baxter algebra $(A,M,P)$.

\medskip

\section{Matching relative Rota-Baxter algebras and matching dendriform algebras}\label{sec3}
In this section, we first introduce matching relative Rota-Baxter algebras and study some of their basic properties. Next, we find some relations between matching relative Rota-Baxter algebras and the notion of matching dendriform algebras introduced in \cite{gao-guo-zhang}.

\begin{defn}   A {\bf matching relative Rota-Baxter algebra} is a triple $(A,M,\{P_x\}_{x\in X})$ consisting of an associative algebra $A$, an $A$-bimodule $M$ and a collection $\{P_x:M\rightarrow A\}_{x\in X}$ of linear maps satisfying
\begin{align}\label{matching-rel-rb-identity}
P_x(u)\cdot P_y(v) = P_x(u\cdot_M P_y(v)) + P_y(P_x(u)\cdot_M v),~~ \mathrm{for} ~u,v \in M ~\mathrm{and} ~x,y \in X.
\end{align}
\end{defn}

\begin{remark}
Let $X$ be a set equipped with some additional structure (e.g. semigroup, group, ring etc.) and let $\tau : X \rightarrow X$ be a set map (not necessarily preserving that additional structure). If $(A, M, \{ P_x \}_{x \in X})$ is a matching relative Rota-Baxter algebra, then it is easy to see that the triple $(A, M, \{ P^\tau_x \}_{x \in X})$ is a new matching relative Rota-Baxter algebra, where $P_x^\tau (u) = P_{\tau (x)} (u)$, for $x \in X$ and $u \in M$. This shows that the identity (\ref{matching-rel-rb-identity}) depends only on the set-theoretic property of $X$.
\end{remark}

\begin{exam} Let $(A,M,P)$ be a relative Rota-Baxter algebra. Then  $(A,M,\{P,-P\})$ is a matching relative Rota-Baxter algebra.
\end{exam}
\begin{exam}  Let $(A,M,P)$ is a relative Rota-Baxter algebra and $\{a_x\}_{x\in X}$ be a collection of elements in the centre of $A$. Then $(A,M,\{P_x\}_{x\in X})$ is a matching relative Rota-Baxter algebra, where
\begin{align*}
P_x(u)~=~P(a_x\cdot_Mu), \text{ for } x\in X \text{ and } u\in M.
\end{align*}
To show this, we observe that
 \begin{align*}
P_x(u)\cdot P_y(v)&=P(a_x\cdot_M u)\cdot P(a_y\cdot_Mv)\\
&=P\big(P(a_x\cdot_Mu)\cdot_M(a_y\cdot_Mv) + (a_x\cdot_Mu)\cdot_MP(a_y\cdot_Mv)\big)\\&=P\big(a_y\cdot(P_x(u)\cdot_Mv) + a_x\cdot(u\cdot_MP_y(v))\big)\\ &=P_y \big( P_x(u)\cdot_Mv \big) + P_x \big( u\cdot_MP_y(v) \big),
\end{align*}
for any $u,v\in M$ and $x,y\in X$. This proves our claim. In particular, if $A$ is a commutative associative algebra and $M$ is a symmetric $A$-bimodule, then for any collection of elements $\{a_x\}_{x\in X}$ in $A$, the triple $(A,M,\{P_x\}_{x\in X})$ is a matching relative Rota-Baxter algebra.
\end{exam}

\begin{exam}  Let $A:=\mathbb{K}[t]$ be the polynomial algebra in one variable $t$. Then $(A,A,\{P_k\}_{k\in {\mathbb{N}\cup \{ 0 \} } })$ is a matching (relative) Rota-Baxter algebra, where $P_k:A\rightarrow A$ is given by $P_k(t^n) =  {\frac{t^{k+n+1}}{k+n+1}}$,  for $k,n\geq 0$. This example can be seen as a particular case of the previous example.\end{exam}

\begin{defn}
 Let $(A,M,\{P_x\}_{x\in X})$ and $(A',M',\{P'_x\}_{x\in X})$ be two matching relative Rota-Baxter algebras. A {\bf morphism} of matching relative Rota-Baxter algebras from  $(A,M,\{P_x\}_{x\in X})$ to $(A',M',\{P'_x\}_{x\in X})$ consists of a pair $(\varphi,\psi)$ of an algebra homomorphism $\varphi :A\rightarrow A'$ and a linear map $\psi :M\rightarrow M'$ satisfying
\begin{center}
$\psi(a\cdot_Mu) = \varphi(a)\cdot_{M'}\psi(u), \quad \psi(u\cdot_Ma) = \psi(u)\cdot_{M'}\varphi(a) \quad \mathrm{and} \quad \varphi\circ P_x = P'_x\circ\psi$,
\end{center}
for all $a\in A, u\in M, x\in X$.
\end{defn}
The collection of all matching relative Rota-Baxter algebras and morphisms between them forms a category, denoted by $\mathbf{mrRBA}_X$. Note that the category of relative Rota-Baxter algebras and the category of matching Rota-Baxter algebras are a subcategory of $\mathbf{mrRBA}_X$.\\

In \cite{aguiar4} Aguiar introduced a notion of the `associative Yang-Baxter equation' (AYBE) as the associative analogue of the classical Yang-Baxter equation. He also observed that a solution of {AYBE} (also called an associative $\textbf{r}$-matrix) gives rise to a Rota-Baxter operator on the underlying algebra. Later, Uchino \cite{uchino} showed that a skew-symmetric associative $\mathbf{r}$-matrix gives rise to a new relative Rota-Baxter operator.

\medskip

Let $A$ be an associative algebra. For any two elements $\mathbf{r} = \sum{r_{(1)}\otimes r_{(2)}}$ and $\mathbf{s} = \sum{s_{(1)}\otimes s_{(2)}}$ of the tensor product $A\otimes A$, we consider the following three elements
\begin{align*}
\mathbf{r}_{13} \mathbf{s}_{12} = \sum{r_{(1)} \cdot s_{(1)}}\otimes & s_{(2)}\otimes r_{(2)} , \quad \mathbf{r}_{12}~ \mathbf{s}_{23} = \sum r_{(1)}\otimes r_{(2)} \cdot s_{(1)} \otimes s_{(2)},\\ &\mathbf{r}_{23}~\mathbf{s}_{13} = \sum s_{(1)}\otimes r_{(1)} \otimes {r_{(2)} \cdot s_{(2)}}
\end{align*}
of $A\otimes A\otimes A$. An element $\mathbf{r} = \sum{r_{(1)}\otimes r_{(2)}} \in A\otimes A$ is called an {\bf associative ${r}$-matrix} if it satisfies the following equation, called the associative Yang-Baxter equation:
\begin{align}
\mathbf{r}_{13} \mathbf{r}_{12} - \mathbf{r}_{12} \mathbf{r}_{23} + \mathbf{r}_{23} \mathbf{r}_{13} = 0.  
\end{align}
In the following, we generalize the notion associative ${r}$-matrix in the context of matching algebras. More precisely, we introduce the following definition.
\begin{defn}  Let $A$ be an associative algebra. A collection $\{ {\bf r}^x=\sum{r^{x}_{(1)} \otimes r^{x}_{(2)}}\}_{x\in X}$ of elements of $A\otimes A$ is said to be a {\bf matching associative ${r}$-matrix} if
\begin{align}
\mathbf{r}^{y}_{13} \mathbf{r}^{x}_{12} - \mathbf{r}^{x}_{12} \mathbf{r}^{y}_{23} + \mathbf{r}^{y}_{23} \mathbf{r}^{x}_{13} = 0, ~\mathrm{for~ all}~x,y~ \in X.\label{x}  
\end{align}
\end{defn}
The equation (\ref{x}) is called the {\bf matching associative Yang-Baxter equation}. Therefore, matching associative $r$-matrices are solutions of the matching associative Yang-Baxter equation.

\medskip

A matching associative $r$-matrix $\{ {\bf r}^x=\sum{r^{x}_{(1)} \otimes r^{x}_{(2)}}\}_{x\in X}$ is said to be {\bf skew-symmetric} if ${\bf r}^x =  - \sigma ({\bf r}^x), \text{ for all } x \in X.$ Here $\sigma : A \otimes A \rightarrow A \otimes A$ is the flip map. Thus,  $\{ {\bf r}^x=\sum{r^{x}_{(1)} \otimes r^{x}_{(2)}}\}_{x\in X}$ is skew-symmetric if and only if
\begin{align*}
  \sum r^x_{(1)} \otimes r^x_{(2)} = - \sum r^x_{(2)} \otimes r^x_{(1)}, \text{ for all } x \in X.
\end{align*}

\medskip

\begin{prop}
Let $A$ be an associative algebra and $\{r^x=\sum{r^{x}_{(1)} \otimes r^{x}_{(2)}}\}_{x\in X}$ be a matching associative ${r}$-matrix. Then for any $A$-bimodule $M$, the triple $(A,M,\{P_x\}_{x\in X})$  is a matching relative Rota-Baxter algebra, where $P_x(u) = \sum r^{x}_{(1)}\cdot_Mu\cdot_Mr^{x}_{(2)},~\mathrm{for}~u\in M.$
\end{prop}

\begin{proof} Note that the identity (\ref{x}) is equivalent to
\begin{align}\label{aybe22}
\sum{r^{y}_{(1)}\cdot r^{x}_{(1)}\otimes r^{x}_{(2)}\otimes r^{y}_{(2)}} - \sum{r^{x}_{(1)}\otimes r^{x}_{(2)}\cdot r^{y}_{(1)}\otimes r^{y}_{(2)}} + \sum{r^{x}_{(1)}\otimes r^{y}_{(1)}\otimes r^{y}_{(2)}\cdot r^{x}_{(2)}} = 0.
\end{align}
In the above three terms, if we replace the first tensor product by $u$ and the second tensor product by $v$, we obtain
\begin{align*}
P_y(P_x(u)\cdot_Mv) - P_x(u)\cdot P_y(v) + P_x(u\cdot P_y(v)) = 0.
\end{align*}
This shows that $(A,M,\{P_x\}_{x\in X})$ is a matching relative Rota-Baxter algebra.
\end{proof}

\begin{prop} Let $A$ be an associative algebra and $\{ {\bf r}^x=\sum{r^{x}_{(1)}\otimes r^{x}_{(2)}}\}_{x\in X}$ be a skew-symmetric matching associative ${r}$-matrix. Then the triple $\big(A,A^{\star},\{P_x:A^{\star}\rightarrow A\}_{x\in X}\big)$ is a matching relative Rota-Baxter algebra, where $P_x:A^{\star}\rightarrow A$ is given by
\begin{align*}
P_x (\alpha) := \sum{\alpha(r^x_{(2)})}~ r^x_{(1)} = -\sum{\alpha(r^x_{(1)})}~r^x_{(2)}, \text{ for }\alpha\in A^{\star},x\in X.  
\end{align*}
Here $A^*$ is equipped with the coadjoint $A$-bimodule structure which is dual to the adjoint $A$-bimodule.
\end{prop}

\begin{proof} For any $\alpha,\beta\in A^{\star}$ and $x,y\in X$, we have
\begin{align}\label{ss1}
 P_x(\alpha)\cdot P_y(\beta) = -\sum{\alpha(r^x_{(1)})}\beta(r^y_{(2)}) ~ r^x_{(2)}\cdot r^y_{(1)}.  
\end{align}
On the other hand, we have
\begin{align}\label{ss2}
& P_x(\alpha\cdot_{A^\star}P_y(\beta)) + P_y(P_x(\alpha)\cdot_{A^\star}\beta) \nonumber \\
&=\sum{P_x\big(\beta(r^y_{(2)})~\alpha\cdot_{A^\star}r^y_{(1)}\big)} - \sum{P_y\big(\alpha(r^x_{(1)}) ~r^x_{(2)}\cdot_{A^\star}\beta\big)} \nonumber \\
&=-\sum{\beta(r^y_{(2)})(\alpha\cdot_{A^\star}r^y_{(1)})(r^x_{(1)})~ r^x_{(2)}}-\sum{\alpha(r^x_{(1)})(r^x_{(2)}\cdot_{A^{\star}}\beta)}(r^x_{(2)})~ r^y_{(1)} \nonumber\\
&=-\sum{\beta(r^y_{(2)})~\alpha(r^y_{(1)}\cdot r^x_{(1)})} ~r^x_{(2)} -\sum{\alpha(r^x_{(1)})~\beta(r^y_{(2)}\cdot r^x_{(2)})~ r^y_{(1)}}.
\end{align}
Note that the matching associative ${r}$-matrix condition (\ref{x}) is equivalent to (\ref{aybe22}). In these tensor products, if we apply $\alpha$ to the first component and $\beta$ to the third component, we get
\begin{align*}
\sum{\alpha(r^y_{(1)}\cdot r^x_{(1)})~\beta(r^y_{(2)})~r^x_{(2)}} - \sum{\alpha(r^x_{(1)}) ~\beta(r^y_{(2)}) ~r^x_{(2)}\cdot r^y_{(1)}} + \sum{\alpha(r^x_{(1)})~\beta(r^y_{(2)}\cdot r^x_{(2)})~r^y_{(1)}} = 0.
\end{align*}
Hence it follows from (\ref{ss1}) and (\ref{ss2}) that
\begin{align*}
P_x(\alpha)\cdot P_y(\beta) =  P_x(\alpha\cdot_{A^\star}P_y(\beta)) + P_y(P_x(\alpha)\cdot_{A^\star}\beta).
\end{align*}
This completes the proof.
\end{proof}

\medskip

The notion of dendriform algebra was first introduced by Loday \cite{loday} in the study of Koszul duality of diassociative algebras. Dendriform algebras are given by two operations satisfying a set of three identities. It turns out that the sum of two operations is associative. Thus, dendriform algebras split associative algebras. In \cite{gao-guo-zhang}, the authors considered matching dendriform algebras as the matching version of dendriform algebras. In the following, we find some relations between matching relative Rota-Baxter algebras and matching dendriform algebras.

\begin{defn}  A {\bf matching dendriform algebra} is a pair $(D,\{\prec_x,\succ_x\}_{x\in X})$ consisting of a $\mathbb{K}$-module $D$ together with a collection $\{\prec_x,\succ_x:D\otimes D\rightarrow D\}_{x\in X}$ of linear maps satisfying
\begin{align}
(a\prec_x b)\prec_y c&~= a\prec_x(b\prec_y c) + a\prec_y(b\succ_x c), \label{mda-1}\\
(a\succ_x b)\prec_y c &~= a\succ_x(b\prec_yc), \label{mda-2}\\
(a\prec_y b)\succ_x c + (a\succ_x b)\succ_y c &~= a\succ_x(b\succ_y c), ~\mathrm{for}~ \mathrm{all}~ a,b,c\in D \text{ and } x,y\in X. \label{mda-3}
\end{align}
\end{defn}

\medskip
A matching dendriform algebra $(D,\{\prec_x,\succ_x\}_{x\in X})$ in which all $\prec_x$ are the same and all $\succ_x$ are the same, is nothing but an ordinary dendriform algebra.

\begin{defn}
Let $(D,\{\prec_x,\succ_x\}_{x\in X})$ and $(D',\{\prec'_x,\succ'_x\}_{x\in X})$ be two matching dendriform algebras. A {\bf morphism} of matching dendriform algebras from $(D,\{\prec_x,\succ_x\}_{x\in X})$ to $(D',\{\prec'_x,\succ'_x\}_{x\in X})$ is given by a linear map $f:D\rightarrow D'$ satisfying
\begin{align*}
f(a\prec_x b) = f(a)\prec'_xf(b)~\mathrm{and}~ f(a\succ_x b) = f(a)\succ'_x f(b),~ \text{for } a,b\in D \text{ and } x\in X.  
\end{align*}
\end{defn}

The collection of all matching dendriform algebras and morphisms between them forms a category, denoted by $\textbf{mDend}_X$. There is a close connection between matching relative Rota-Baxter algebras and matching dendriform algebras which we will describe now.

\begin{prop}  (i) Let  $(A,M,\{P_x\}_{x\in X})$ be a matching relative Rota-Baxter algebra. Then the pair $(M,\{\prec_x,\succ_x\}_{x\in X})$ is a matching dendriform algebra, where the operations $\prec_x,\succ_x:M\otimes M\rightarrow M$ are given by
\begin{align*}
u\prec_x v=u\cdot_MP_x(v)~~ \text{ and } ~~u\succ_xv=P_x(u)\cdot_Mv,~\text{for}~u,v\in M, x\in X.    
\end{align*}

(ii) If $(A,M,\{P_x\}_{x\in X})$ and $(A',M',\{P'_x\}_{x\in X})$ are two matching relative Rota-Baxter algebras, and $(\varphi,\psi)$ is a morphism between them, then the linear map $\psi : M\rightarrow M'$ is a morphism of induced matching dendriform algebras from  $(M,\{\prec_x,\succ_x\}_{x\in X})$ to $(M',\{\prec'_x,\succ'_x\}_{x\in X})$.
\end{prop}

\begin{proof}
(i) For any $u,v,w\in M$ and $x,y\in X$, we have
\begin{align*}
(u\prec_xv)\prec_yw&= (u\cdot_M P_x(v))~._M~P_y(w)\\
&=~u\cdot_M(P_x(v)\cdot P_y(w))\\
&=~u\cdot_MP_x(v\cdot_MP_y(w)) ~+~ u\cdot_MP_y(P_x(v)\cdot_Mw)\\
&=~u\prec_x(v\prec_yw)~+~u\prec_y(v\succ_xw).
\end{align*}
Similarly,
\begin{align*}
(u\succ_xv)\prec_yw =(P_x(u)\cdot_Mv)\cdot_MP_y(w) = P_x(u)\cdot_M(v\cdot_MP_y(w))= u\succ_x(v\prec_yw),
\end{align*}
\begin{align*}
(u\prec_yv)\succ_xw~+~(u\succ_xv)\succ_yw &=~P_x(u\cdot_MP_y(v))\cdot_Mw~+~P_y(P_x(u)\cdot_Mv)\cdot_Mw\\
&=~(P_x(u)\cdot P_y(v))\cdot_Mw\\
&=~P_x(u)\cdot_M(P_y(v)\cdot_Mw)=u\succ_x (v\succ_y w).
\end{align*}
This shows that $(M,\{\prec_x,\succ_x\}_{x\in X})$ is a matching dendriform algebra.

\medskip

(ii) For any $u,v\in M$ and $x\in X$, we have
\begin{align*}
 \psi(u\prec_xv) = \psi(u\cdot_MP_x(v)) = \psi(u)\cdot_{M'}\varphi(P_x(v)) = \psi(u)\cdot_{M'}P'_x(\psi(v)) = \psi(u)\prec'_x\psi(v),\\
 \psi(u\succ_xv) = \psi(P_x(u)\cdot_Mu) = \varphi(P_x(u))\cdot_{M'} \psi(v) = P'_x (\psi(u))\cdot_{M'}\psi(v) = \psi(u)\succ'_x\psi(v).
\end{align*}
This shows that $\psi$ is a morphism of matching dendriform algebras.
\end{proof}

\medskip

It follows from the above proposition that there is a functor $\mathscr{F}:\mathbf{mrRBA}_X\rightarrow \mathbf{mDend}_X$  from the category of matching relative Rota-Baxter algebras to the category of matching dendriform algebras. In the following, we will construct a functor in the other direction. We start with the following result.

\begin{prop} Let $(D,\{\prec_x,\succ_x\}_{x\in X})$ be a matching dendriform algebra. Then $(D\otimes {\mathbb{K}[X]},\prec,\succ)$ is an ordinary dendriform algebra, where
\begin{align*}
(a\otimes x)\prec(b\otimes y) := (a\prec_yb)\otimes x~\mathrm{and}~(a\otimes x)\succ(b\otimes y) := (a\succ_xb)\otimes y,
\end{align*}
for $a\otimes x,~b\otimes y \in D\otimes {\mathbb{K}[X]}$.
\end{prop}

\begin{proof} For any $a\otimes x,b\otimes y~ \mathrm{and}~ c\otimes z \in D\otimes {\mathbb{K}[X]}$, we have
\begin{align*}
\big((a\otimes x)\prec(b\otimes y)\big)\prec(c\otimes z) &= \big((a\prec_yb)\otimes x\big)\prec(c\otimes z)\\
&=\big((a\prec_yb)\prec_zc\big)\otimes x\\
&=\big(a\prec_y(b\prec_zc)\big)\otimes x+\big(a\prec_z(b\succ_yc)\big)\otimes x\\
&=(a\otimes x)\prec\big((b\prec_zc)\otimes y+(b\succ_yc)\otimes{z})\big)\\
&=(a\otimes x)\prec \big((b\otimes y)\prec(c\otimes z)+(b\otimes y)\succ(c\otimes z)\big).
\end{align*}
Similarly, we can verify that
\begin{align*}
\big((a\otimes x)\succ(b\otimes y)\big)\prec(c\otimes z)=(a\otimes x)\succ \big((b\otimes y))\prec(c\otimes z)\big),
\end{align*}
\begin{align*}
\big((a\otimes x)\prec(b\otimes y) ~+~ (a\otimes x)\succ(b\otimes y)\big)\succ(c\otimes z) = (a\otimes x)\succ\big((b\otimes y))\succ(c\otimes z)\big).
\end{align*}
This completes the proof.
\end{proof}

Let $(D,\{\prec_x,\succ_x\}_{x\in X})$ be a matching dendriform algebra. Then it follows from the above proposition that the space $D\otimes {\mathbb{K}[X]}$ inherits an associative algebra structure with the multiplication given by
\begin{align*}
(a\otimes x) \bullet (b\otimes y) =(a\prec_yb)\otimes{x} ~+~ (a\succ_xb)\otimes{y},   \text{ for } (a\otimes x),(b\otimes y)\in D\otimes {\mathbb{K}[X]}.
\end{align*}
Moreover, there is a $(D\otimes {\mathbb{K}[X]})$-bimodule structure on $D$ with the left and right action maps
\begin{align*}
(a\otimes x)\cdot_Db=a\succ_xb ~ \text{ and } ~ b\cdot_D(a\otimes x)=b\prec_xa, \text{ for } a\otimes x\in D\otimes {\mathbb{K}[X]},~b\in D.
\end{align*}
With all these notations, the triple $(D\otimes {\mathbb{K}[X]},D,\{\mathrm{id}_x \}_{x\in X})$ is a matching relative Rota-Baxter algebra, where the map $\mathrm{id}_x:D\rightarrow D\otimes {\mathbb{K}[X]}$ is given by $\mathrm{id}_x(a)=a\otimes{x}$, for $x\in X$. Moreover, the induced matching dendriform algebra structure on $D$ coincides with the given one. Thus, we have constructed a matching relative Rota-Baxter algebra from a matching dendriform algebra. This construction is also functorial. Hence we obtain a functor $\mathscr{G}:\mathbf{mDend}_X\rightarrow \mathbf{mrRBA}_X$. In the following result, we show that the functor $\mathscr{G}$ is left adjoint to the functor $\mathscr{F}$. More explicitly, we have the following result.

\begin{prop}  For any matching dendriform algebra $(D,\{\prec_x,\succ_x\}_{x\in X})$ and a matching relative Rota-Baxter algebra $(A,M,\{P_x\}_{x\in X})$, we have
\begin{align*}
\mathrm{Hom}_{\mathbf{mrRBA}_X} &\big(( D\otimes {\mathbb{K}[X]},D,\{ \mathrm{id}_x\}_{x\in X}), (A,M,\{P_x\}_{x\in X})\big) \\  &\cong \mathrm{Hom}_{\mathbf{mDend}_X} \big((D,~\{\prec_x,\succ_x\}_{x\in X}),(M,\{\prec_x,\succ_x\}_{x\in X})\big).
\end{align*}
\end{prop}

\begin{proof} Let $\psi \in \mathrm{Hom}_{\mathbf{mDend}_X} \big((D,\{\prec_x,\succ_x\}_{x\in X}),(M,\{\prec_x,\succ_x\}_{x\in X})\big)$. Then we define a map
\begin{align*}
P^{\psi}:D\otimes {\mathbb{K}[X]}\rightarrow A ~\text{ by }~ P^{\psi}(a\otimes{x})=P_x(\psi(a)), \text{ for } a\otimes{x} \in D\otimes {\mathbb{K}[X]}.
\end{align*}
It is easy to see that $P^{\psi}$ is a morphism of associative algebras that satisfies
\begin{align*}
\psi((a\otimes x)\cdot_Db) = P^{\psi}(a\otimes x)\cdot_M\psi(b), \quad
\psi(b\cdot_D(a\otimes x)) = \psi(b)\cdot_MP^{\psi}(a\otimes x)
~~\mathrm{ and }~~ P^{\psi}\circ \mathrm{id}_x = P_x\circ \psi,
\end{align*}
for $a\otimes x\in D\otimes {\mathbb{K}[X]},~b\in D$ and $x\in X$. This shows that
\begin{align*}
(P^{\psi},\psi)\in \mathrm{Hom}_{\mathbf{mrRBA}_X}\big(( D\otimes {\mathbb{K}[X]},D,\{ \mathrm{id}_x\}_{x\in X}), (A,M,\{P_x\}_{x\in X})\big).
\end{align*}
Conversely, if we have a morphism $(\varphi,\psi)\in \mathrm{Hom}_{\mathbf{mrRBA}_X}\big(( D\otimes {\mathbb{K}[X]},D,\{ \mathrm{id}_x\}_{x\in X}), (A,M,\{P_x\}_{x\in X})\big)$, then we have seen that $\psi \in \mathrm{Hom}_{\mathbf{mDend}_X} \big((D,\{\prec_x,\succ_x\}_{x\in X}),(M,\{\prec_x,\succ_x\}_{x\in X})\big).$ This proves the desired result.
\end{proof}

\begin{remark}\label{embedding}
Let $(D, \{ \prec_x, \succ_x \}_{x \in X})$ be a matching dendriform algebra. Then we have seen that the triple $(D \otimes \mathbb{K}[X] , D, \{ \mathrm{id}_x \}_{x \in X} )$ is a matching relative Rota-Baxter algebra. As a consequence, the pair $\big( (D \otimes \mathbb{K}[X]) \oplus D, \{ P_x \}_{x \in X} \big)$ is a matching Rota-Baxter algebra, where the associative multiplication on $ (D \otimes \mathbb{K}[X]) \oplus D$ is given by the semi-direct product
\begin{align*}
    (a \otimes x, a')  \ltimes (b \otimes y, b') = \big(  (a \prec_y b) \otimes x + (a \succ_x b) \otimes y,~ a \succ_x b' + a' \prec_y b  \big),
\end{align*}
and $P_x : (D \otimes \mathbb{K}[X]) \oplus D \rightarrow (D \otimes \mathbb{K}[X]) \oplus D$ is given by $P_x (b \otimes y , b') = (b' \otimes x, 0)$, for $(b \otimes y, b') \in (D \otimes \mathbb{K}[X]) \oplus D$ and $x \in X$. Therefore, the space $(D \otimes \mathbb{K}[X]) \oplus D$ induces a matching dendriform algebra structure. With this notation, the inclusion map $D \hookrightarrow (D \otimes \mathbb{K}[X]) \oplus D$, $a \mapsto (0, a)$ is a morphism of matching dendriform algebras. Hence it is an embedding of the given matching dendriform algebra $(D, \{ \prec_x, \succ_x \}_{x \in X})$ into the matching Rota-Baxter algebra  $\big( (D \otimes \mathbb{K}[X]) \oplus D, \{ P_x \}_{x \in X} \big)$.
\end{remark}

\medskip

\section{Cohomology of matching relative Rota-Baxter algebras}\label{sec4}
In this section, we introduce the cohomology of a matching relative Rota-Baxter algebra $(A,M,\{P_x\}_{x\in X})$. This cohomology is obtained as a byproduct of the Hochschild cohomology of $A$ and certain cohomology induced by the family of operators $\{P_x\}_{x\in X}$. As a particular case, we define the cohomology of a matching Rota-Baxter algebra $(A,\{P_x\}_{x\in X})$.

Let $A$ be an associative algebra and $M$ be an $A$-bimodule. For any $n\geq 1$, let $\mathfrak{g}^n$ be the set of all elements of the form $P :=\displaystyle \oplus_{{x_1,\ldots,x_n}\in X} P_{x_1,\ldots,x_n}$ with ${P_{x_1,\ldots,x_n}}\in \mathrm{Hom}(M^{\otimes n},A)$. Thus, an element $P\in \mathfrak{g}^n$ is given by a collection $\{P_{x_1,\ldots,x_n} : M^{\otimes n}\rightarrow A\}_{{x_1,\ldots,x_n}\in X}$ of maps labelled by the elements of $X^{\times n}$. Note that $\mathfrak{g}^n$ is a $\mathbb{K}$-module with the abelian group structure and the $\mathbb{K}$-action are respectively given by
\begin{align*}
P+P':= \displaystyle \oplus_{{x_1,\ldots,x_n}\in X} (P_{x_1,\ldots,x_n} + P'_{x_1,\ldots,x_n}) ~~\mathrm{~ and ~}~~ \lambda P := \displaystyle \oplus_{{x_1,\ldots,x_n}\in X} \lambda P_{x_1,\ldots,x_n},
\end{align*}
for $P,P'\in \mathfrak{g}^n$ and $\lambda \in \mathbb{K}$. We define $\mathfrak{g}^0 = A$. With the above notations, we have the following result.

\begin{thm} Let $A$ be an associative algebra and $M$ be an $A$-bimodule. Then the graded $\mathbb{K}$-module $\displaystyle \oplus_{n\geq 0}\mathfrak{g}^n$ inherits a graded Lie algebra structure with the bracket given by

\begin{align}\label{derived-bracket}
&{\llbracket P,Q\rrbracket }_{x_1,\ldots,x_{m+n}}(u_1,\ldots,u_{m+n})\\
&=~\displaystyle\sum_{i=1}^{m} (-1)^{(i-1)n} {P_{x_1,\ldots,x_{i-1},x_{i+n},\ldots,x_{m+n}}}\big(u_1,\ldots,u_{i-1},Q_{x_i,\ldots,x_{i+n-1}}(u_i,\ldots,u_{i+n-1}) \cdot_M u_{i+n},\ldots,u_{m+n}\big) \nonumber \\
&-\displaystyle\sum_{i=1}^{m}(-1)^{in} {P_{x_1,\ldots,x_{i},x_{i+n+1},\ldots,x_{m+n}}} \big(u_1,\ldots,u_{i-1},u_i\cdot_M Q_{x_{i+1},\ldots,x_{i+n}} (u_{i+1},\ldots,u_{i+n}),u_{i+n+1},\ldots,u_{m+n}\big)  \nonumber \\
&-(-1)^{mn}\bigg\{\displaystyle\sum_{i=1}^{n} (-1)^{(i-1)m}  {Q_{x_1,\ldots,x_{i-1},x_{i+m},\ldots,x_{m+n}}} \big(u_1,\ldots,u_{i-1},P_{x_i,\ldots,x_{i+m-1}} (u_i,\ldots,u_{i+m-1})\cdot_Mu_{i+m},\ldots \big)  \nonumber  \\
&-\displaystyle\sum_{i=1}^{n}(-1)^{im} {Q_{x_1,\ldots,x_{i},x_{i+m+1},\ldots,x_{m+n}}}\big(u_1,\ldots,u_{i-1},u_i \cdot_M P_{x_{i+1},\ldots,x_{i+m}}(u_{i+1},\ldots,u_{i+m}),u_{i+m+1},\ldots,u_{m+n}\big)\bigg\}  \nonumber \\
&+(-1)^{mn}\big[ P_{x_1,\ldots,x_m}(u_1,\ldots,u_m)\cdot Q_{x_{m+1},\ldots,x_{m+n}}(u_{m+1},\ldots,u_{m+n})  \nonumber \\
&-~(-1)^{mn}Q_{x_1,\ldots,x_n}(u_1,\ldots,u_n)\cdot P_{x_{n+1},\ldots,x_{m+n}}(u_{n+1},\ldots,u_{m+n})\big],  \nonumber
\end{align}
\begin{align}
\llbracket P,a\rrbracket_{x_1,\ldots,x_m}(u_1,\ldots,u_m) =& \displaystyle\sum_{i=1}^{m} P_{x_1,\ldots,x_m}(u_1,\ldots,u_{i-1},a\cdot_Mu_i-u_i\cdot_Ma,u_{i+1},\ldots,u_m)  \\ \medskip
&+  P_{x_1,\ldots,x_m}(u_1,\ldots,u_m)\cdot a - a\cdot P_{x_1,\ldots,x_m}(u_1,\ldots,u_m),  \nonumber \\
\llbracket a,b\rrbracket =~& a\cdot b - b\cdot a,
\end{align}
for $P =\displaystyle \oplus_{{x_1,\ldots,x_m}\in X} P_{x_1,\ldots,x_m} \in \mathfrak{g}^m$,  $Q=\displaystyle \oplus_{{x_1,\ldots,x_n}\in X} P_{x_1,\ldots,x_n}\in \mathfrak{g}^n$, $a,b\in A$ and $u_1,\ldots,u_{m+n}\in M$.
\end{thm}

\begin{proof}
For any $m,n \geq 0$, we define a linear map $\diamond : \mathfrak{g}^m \otimes \mathfrak{g}^n \rightarrow \mathfrak{g}^{m+n}$ by
\begin{align*}
   & (P \diamond Q)_{x_1, \ldots, x_{m+n}} (u_1, \ldots, u_{m+n} ) \\
     &=~\displaystyle\sum_{i=1}^{m} (-1)^{(i-1)n} {P_{x_1,\ldots,x_{i-1},x_{i+n},\ldots,x_{m+n}}}\big(u_1,\ldots,u_{i-1},Q_{x_i,\ldots,x_{i+n-1}}(u_i,\ldots,u_{i+n-1}) \cdot_M u_{i+n},\ldots,u_{m+n}\big) \nonumber \\
&~-\displaystyle\sum_{i=1}^{m}(-1)^{in} {P_{x_1,\ldots,x_{i},x_{i+n+1},\ldots,x_{m+n}}} \big(u_1,\ldots,u_{i-1},u_i\cdot_M Q_{x_{i+1},\ldots,x_{i+n}} (u_{i+1},\ldots,u_{i+n}),u_{i+n+1},\ldots,u_{m+n}\big)  \nonumber \\
&~+ (-1)^{mn} P_{x_1,\ldots,x_m}(u_1,\ldots,u_m)\cdot Q_{x_{m+1},\ldots,x_{m+n}}(u_{m+1},\ldots,u_{m+n}),
\end{align*}
for $P = \displaystyle \oplus_{{x_1,\ldots,x_m}\in X} P_{x_1,\ldots,x_m} \in \mathfrak{g}^m$ and $Q=\displaystyle \oplus_{{x_1,\ldots,x_n}\in X} P_{x_1,\ldots,x_n}\in \mathfrak{g}^n$. It is easy to verify that the operation $\diamond$ makes the graded space $\oplus_{n \geq 0} \mathfrak{g}^n$ into a graded pre-Lie algebra. As a consequence, the operation
\begin{align*}
    \llbracket P, Q \rrbracket = P \diamond Q - (-1)^{mn} Q \diamond P, \text{ for } P \in \mathfrak{g}^m, Q \in \mathfrak{g}^n
\end{align*}
makes the pair $(\oplus_{n \geq 0} \mathfrak{g}^n, \llbracket ~, ~\rrbracket )$ into a graded Lie algebra.
\end{proof}

\begin{thm}\label{thm-mc-char}  Let $A$ be an associative algebra, $M$ be an $A$-bimodule and $\{P_x:M\rightarrow A\}_{x\in x}$ be a collection of linear maps. Then the triple $(A,M,\{P_x\}_{x\in X})$ is a matching relative Rota-Baxter algebra if and only if the element $P = \displaystyle \oplus_{x\in X}P_x\in \mathfrak{g}^1$ is a Maurer-Cartan element in the graded Lie algebra $(\displaystyle \oplus_{n\geq 0}\mathfrak{g}^n,\llbracket~,~\rrbracket)$.
\end{thm}

\begin{proof} For any $u,v\in M~and~x,y\in X$, it follows from (\ref{derived-bracket}) that
\begin{align*}
\llbracket P,P\rrbracket_{x,y}(u,v) = 2 \big\{P_y(P_x(u)\cdot_Mv) +  P_x(u\cdot_MP_y(v)) - P_x(u)\cdot P_y(v) \big\}.
\end{align*}
This shows that $\llbracket P,P\rrbracket = 0$ if and only if the collection $\{P_x\}_{x\in X}$ satisfies (\ref{matching-rel-rb-identity}). In other words, $P = \displaystyle \oplus_{x\in X} P_x$ is a Maurer-Cartan element in the graded Lie algebra $(\displaystyle \oplus_{n\geq 0}\mathfrak{g}^n,\llbracket~,~\rrbracket)$ if and only if $(A,M,\{P_x\}_{x\in X})$ is a matching relative Rota-Baxter algebra.
\end{proof}

Let $(A,M,\{P_x\}_{x\in X})$ be a matching relative Rota-Baxter algebra. Then it follows from Theorem \ref{thm-mc-char} that $P = \displaystyle \oplus_{x\in X} P_x$ is a Maurer-Cartan element in the graded Lie algebra $(\displaystyle \oplus_{n\geq 0}\mathfrak{g}^n,\llbracket~,~\rrbracket)$. Hence the element $P$ induces a differential $d_P = \llbracket P,~\rrbracket:\mathfrak{g}^n\rightarrow \mathfrak{g}^{n+1}$, for $n\geq 0$. Moreover, the differential $d_P$ makes the triple $(\displaystyle \oplus_{n\geq 0}\mathfrak{g}^n,\llbracket~,~\rrbracket,d_P)$  into a differential graded Lie algebra. The importance of this differential graded Lie algebra is given by the following result.

\begin{prop} Let $(A,M,\{P_x\}_{x\in X})$ be a matching relative Rota-Baxter algebra. Then for any collection $\{P'_x:M\rightarrow A\}_{x\in X}$  of linear maps, the triple $(A,M,\{P_x + P'_x\}_{x\in X})$ is also a matching relative Rota-Baxter algebra if and only if the element $P'=\displaystyle \oplus_{x\in X}P'_x\in \mathfrak{g}^1$ is a Maurer-Cartan element in the differential graded Lie algebra $\big(\displaystyle \oplus_{n\geq 0}\mathfrak{g}^n,\llbracket~,~\rrbracket,d_P\big)$.
\end{prop}

\begin{proof} Note that $(A,M,\{P_x + P'_x\}_{x\in X})$ is a matching relative Rota-Baxter algebra if and only if $P + P'\in \mathfrak{g}^1$ is a Maurer-Cartan element in the graded Lie algebra $(\displaystyle \oplus_{n\geq 0}\mathfrak{g}^n ,\llbracket~,~\rrbracket)$. This holds if and only if
\begin{align*}
\llbracket {P+P'},{P+P'}\rrbracket = 0 &\iff  \llbracket P,P'\rrbracket + \llbracket P',P\rrbracket + \llbracket P',P'\rrbracket = 0\\
&\iff d_P(P') + \frac{1}{2} \llbracket P',P'\rrbracket = 0
\end{align*}
which is equivalent to the fact that $P'$ is a Maurer-Cartan element in the differential graded Lie algebra $\big(\displaystyle \oplus_{n\geq 0}\mathfrak{g}^n,\llbracket~,~\rrbracket,d_P\big)$. This completes the proof.
\end{proof}

Let $(A,M,\{P_x\}_{x\in X})$ be a matching relative Rota-Baxter algebra. Consider the Maurer-Cartan element $P = \displaystyle \oplus_{x\in X}P_x$  in the graded Lie algebra $(\displaystyle \oplus_{n\geq 0}\mathfrak{g}^n,\llbracket~,~\rrbracket)$. We have seen that $P$ induces a differential $d_P$ on the graded vector space  $\displaystyle \oplus_{n\geq 0}\mathfrak{g}^n$, hence $P$ induces a cohomology theory. To make it more explicit, for each $n \geq 0$, we define
\begin{align*}
C^n_{\{P_x\}}(M,A)=\mathfrak{g}^n=\big\{f= \displaystyle \oplus_{{x_1,\ldots,x_n}\in X} f_{x_1,\ldots,x_n}~|~ f_{x_1,\ldots,x_n} \in \mathrm{Hom}(M^{\otimes{n}},A) \text{ for all } x_1,\ldots,x_n\in X \big\}
\end{align*}
and a differential $\delta_{\{P_x\}} : C^n_{\{P_x\}}(M,A)\rightarrow C^{n+1}_{\{P_x\}}(M,A)$ by
\begin{align*}
\delta_{\{P_x\}}(f) := (-1)^n \llbracket P,f\rrbracket, \text{ for } f \in C^n_{\{P_x\}}(A,M).
\end{align*}
The differential $\delta_{\{P_x\}}$ is explicitly given by
\begin{align}
&(\delta_{\{P_x\}}(f))_{x_1,\ldots,x_{n+1}}(u_1,\ldots,u_{n+1})\\
&= P_{x_1}(u_1)\cdot f_{x_2,\ldots,x_{n+1}}(u_2,\ldots,u_{n+1}) - P_{x_1}(u_1\cdot_Mf_{x_2,\ldots,x_{n+1}}(u_2,\ldots,u_{n+1}) )\nonumber \\
&+\displaystyle\sum_{i=1}^{n}(-1)^i~ f_{x_1,\ldots,x_i,x_{i+2},\ldots,x_{n+1}}(u_1,\ldots,u_{i-1},u_i\cdot_MP_{x_{i+1}}(u_{i+1}),\ldots,u_{n+1}) \nonumber \\
&+\displaystyle\sum_{i=1}^{n}(-1)^i~ f_{x_1,\ldots,x_{i-1},x_{i+1},\ldots,x_{n+1}} (u_1,\ldots,u_{i-1},P_{x_i}(u_i)\cdot_M u_{i+1},\ldots,u_{n+1}) \nonumber \\
&+ (-1)^{n+1}~f_{x_1,\ldots,x_{n}}(u_1,\ldots,u_{n})\cdot P_{x_{n+1}}(u_{n+1}) - (-1)^{n+1}~P_{x_{n+1}}(f_{x_1,\ldots,x_{n}}(u_1,\ldots,u_{n})\cdot_Mu_{n+1}),  \nonumber
\end{align}
for $f\in C^n_{\{P_x\}}(M,A),~u_1,\ldots,u_{n+1}\in M$ and $x_1,\ldots,x_{n+1}\in X$. We denote the corresponding cohomology groups by $H^{\bullet}_{\{P_x\}}(M,A)$.

\begin{remark}  Note that the cohomology groups  $H^{\bullet}_{\{P_x\}}(M,A)$ are associated with the collection $\{P_x\}_{x\in X}$ of operators. Hence such cohomology groups are only useful to study the operators $\{P_x\}_{x\in X}$, but not the whole matching relative Rota-Baxter algebra $(A,M,\{P_x\}_{x\in X})$.
\end{remark}

In the following, we define another cohomology that captures the full information about a matching relative Rota-Baxter algebra $(A,M,\{P_x\}_{x\in X})$. This cohomology will generalise the cohomology of a relative Rota-Baxter algebra defined in \cite{das-mishra}.

Let $(A,M,\{P_x\}_{x\in X})$ be a matching relative Rota-Baxter algebra. For each $n\geq 0$, we define a $\mathbb{K}$-module $C^n_{\mathrm{mrRBA}}(A,M,\{P_x\}_{x\in X})$ as follows:
\begin{align*}
    C^n_{\mathrm{mrRBA}}(A,M,P)=
\begin{cases}
    0,& \text{if } n= 0\\
    \mathrm{Hom}(A,A)\oplus \mathrm{Hom}(M,M),              & \text{if } n = 1\\
    \mathrm{Hom}(A^{\otimes n},A)\oplus \mathrm{Hom}(\mathcal{A}^{n-1,1},M)\oplus C_{\{P_x\}}^{n-1}(M,A),              & \text{if } n \geq 2.
\end{cases}
\end{align*}
Moreover, we define a map $\delta_{\mathrm{mrRBA}}:C^n_{\mathrm{mrRBA}}(A,M,\{P_x\}_{x\in X})\rightarrow C^{n+1}_{\mathrm{mrRBA}}(A,M,\{P_x\}_{x\in X})$ by
\begin{align*}
\delta_{\mathrm{mrRBA}}(\alpha,\beta) &= \big(\delta_{\mathrm{Hoch}}(\alpha), \delta^\alpha_{\mathrm{Hoch}}(\beta), h_{\{P_x\}}(\alpha,\beta)\big) ,\\
\delta_{\mathrm{mrRBA}}(\alpha,\beta,\gamma) &= \big(\delta_{\mathrm{Hoch}}(\alpha), \delta^\alpha_{\mathrm{Hoch}}(\beta), \delta_{\{P_x\}}(\gamma) + h_{\{P_x\}}(\alpha,\beta)\big),
\end{align*}
for any  $(\alpha,\beta)\in C^1_{\mathrm{mrRBA}}(A,M,\{P_x\})$ and $(\alpha,\beta,\gamma)\in C^{n\geq 2}_{\mathrm{mrRBA}}(A,M,\{P_x\})$. Here the map
\begin{align*}
h_{ \{ P_x \} }:\mathrm{Hom}(A^{\otimes n},A)\oplus \mathrm{Hom}(\mathcal{A}^{n-1,1},M)\rightarrow C^n_{\{P_x\}}(M,A) \text{ is given by }
\end{align*}
\begin{align*}
&\big( h_{\{P_x\}} (\alpha,\beta) \big)_{x_1,\ldots,x_n} (u_1, \ldots, u_n) \\
&= (-1)^n\big\{\alpha\big(P_{x_1}(u_1),\ldots,P_{x_n}(u_n)\big) - \sum_{i=1}^{n}   P_{x_i} \beta \big(P_{x_1}(u_1),\ldots,u_i,\ldots,P_{x_n}(u_n) \big) \big\},
\end{align*}
for $u_1,\ldots,u_n \in M$ and $x_1,\ldots,x_n \in M$.

\begin{thm}
 Let $(A,M,\{P_x\}_{x\in X})$ be a matching relative Rota-Baxter algebra. With the above notations, the pair $\big\{ C^\bullet_{\mathrm{mrRBA}}(A,M,\{P_x\}_{x\in X}) , \delta_{\mathrm{mrRBA}} \big\}$ is a cochain complex.
 \end{thm}

\begin{proof}
For $(\alpha, \beta) \in \mathrm{Hom} (A^{\otimes n}, A) \oplus \mathrm{Hom}(\mathcal{A}^{n-1,1}, M)$, we have
\begin{align*}
    &\bigg[  \delta_{ \{ P_x \}} \big( h_{\{ P_x \}} (\alpha, \beta) \big) ~+~  h_{\{ P_x \}} \big(  \delta_\mathrm{Hoch} (\alpha), \delta_\mathrm{Hoch}^\alpha (\beta) \big) \bigg]_{x_1, \ldots, x_{n+1}} (u_1, \ldots, u_{n+1}) \\
    &= P_{x_1} \cdot \big(   h_{\{ P_x \}} (\alpha, \beta) \big)_{x_2, \ldots, x_{n+1}} (u_2, \ldots, u_{n+1}) ~-~ P_{x_1} \big( u_1 \cdot_M (   h_{\{ P_x \}} (\alpha, \beta))_{x_2, \ldots, x_{n+1}} (u_2, \ldots, u_{n+1})   \big) \\
   & + \sum_{i=1}^n (-1)^i \big(   h_{\{ P_x \}} (\alpha, \beta) \big)_{x_1, \ldots, x_i, x_{i+2}, \ldots, x_{n+1}} \big(  u_1, \ldots, u_{i-1}, u_i \cdot_M P_{x_{i+1}} (u_{i+1}), u_{i+2}, \ldots, u_{n+1} \big) \\
   &  + \sum_{i=1}^n (-1)^i \big(   h_{\{ P_x \}} (\alpha, \beta) \big)_{x_1, \ldots, x_{i-1}, x_{i+1}, \ldots, x_{n+1}} \big(  u_1, \ldots, u_{i-1}, P_{x_i} (u_i) \cdot_M u_{i+1}, u_{i+2}, \ldots, u_{n+1} \big) \\
   &  + (-1)^{n+1}  \big\{ \big(   h_{\{ P_x \}} (\alpha, \beta) \big)_{x_1, \ldots, x_n} (u_1, \ldots, u_n) \cdot P_{x_{n+1}} (u_{n+1}) \\
   &  \qquad \qquad \qquad \qquad - P_{x_{n+1}} \big(  ( h_{\{ P_x \}} (\alpha, \beta)  )_{x_1, \ldots, x_n} (u_1, \ldots, u_n) \cdot_M u_{n+1} \big) \big\} \\
   &  + (-1)^{n+1} \bigg\{ \big(  \delta_\mathrm{Hoch} (\alpha) \big) \big( P_{x_1}(u_1), \ldots, P_{x_{n+1}} (u_{n+1})  \big) \\
   & \qquad \qquad \qquad \qquad  - \sum_{r=1}^{n+1} P_{x_r}  \big(  (\delta^\alpha_\mathrm{Hoch} (\beta)) (P_{x_1} (u_1), \ldots, u_r, \ldots, P_{x_{n+1}} (u_{n+1}) )  \big)  \bigg\}\\
    & = (-1)^n \bigg\{ P_{x_1} (u_1) \cdot \alpha \big( P_{x_2} (u_2), \ldots, P_{x_{n+1}} (u_{n+1})  \big) \\
   & \qquad \qquad \qquad \qquad  - \sum_{s=1}^n P_{x_1} (u_1) \cdot P_{x_{s+1}} \beta \big(  P_{x_2}(u_2), \ldots, u_{s+1}, \ldots, P_{x_{n+1}} (u_{n+1})  \big) \bigg\} \\
   &  - (-1)^n \bigg\{  P_{x_1} \big( u_1 \cdot_M \alpha \big( P_{x_2} (u_2), \ldots, P_{{x_{n+1}}} (u_{n+1})  \big)  \big) \\
   & \qquad \qquad \qquad \qquad  - \sum_{s=1}^n P_{x_1} \big(   u_1 \cdot_M P_{x_{s+1}} \beta \big( P_{x_2} (u_2), \ldots, u_{s+1}, \ldots, P_{x_{n+1}} (u_{n+1})  \big) \big)  \bigg\} \\
    & + (-1)^n \sum_{i=1}^n (-1)^i \alpha \big(  P_{x_1} (u_1), \ldots, P_{x_{i-1}} (u_{i-1}), P_{x_i} (u_i \cdot_M P_{x_{i+1}} (u_{i+1})), P_{x_{i+2}} (u_{i+2}), \ldots, P_{x_{n+1}} (u_{n+1})  \big) \\
    & - (-1)^n \sum_{i=1}^n (-1)^i \sum_{r=1, r \neq i+1}^{n+1} P_{x_r} \beta \big(\underbrace{ P_{x_1} (u_1), \ldots, P_{x_i} (u_i \cdot_M P_{x_{i+1}} (u_{i+1}) ), \ldots, P_{x_{n+1}} (u_{n+1}) }_{P_{x_r} \text{ is missing}}  \big) \\
    & + (-1)^n \sum_{i=1}^n (-1)^i \alpha \big(  P_{x_1} (u_1), \ldots, P_{x_{i-1}} (u_{i-1}), P_{x_{i+1}} (u_i \cdot_M P_{x_{i+1}} (u_{i+1})), P_{x_{i+2}} (u_{i+2}), \ldots, P_{x_{n+1}} (u_{n+1})  \big) \\
    & - (-1)^n \sum_{i=1}^n (-1)^i \sum_{r=1, r \neq i}^{n+1} P_{x_r} \beta \big(\underbrace{ P_{x_1} (u_1), \ldots, P_{x_{i+1}} (u_i \cdot_M P_{x_{i+1}} (u_{i+1}) ), \ldots, P_{x_{n+1}} (u_{n+1}) }_{P_{x_r} \text{ is missing}}  \big) \\
    & + (-1)^{n+1} (-1)^n \bigg\{  \alpha \big( P_{x_1} (u_1), \ldots, P_{x_n} (u_n)  \big) - \sum_{s=1}^n P_{x_s} \beta \big( P_{x_1} (u_1), \ldots, u_s, \ldots, P_{x_n} (u_n) \big) \bigg\} \cdot P_{x_{n+1}} (u_{n+1}) \\
    & -(-1)^{n+1} (-1)^n P_{x_{n+1}} \bigg( \big(   \alpha \big( P_{x_1} (u_1), \ldots, P_{x_n} (u_n)  \big) - \sum_{s=1}^n P_{x_s} \beta \big( P_{x_1} (u_1), \ldots, u_s, \ldots, P_{x_n} (u_n) \big) \big) \cdots_M u_{n+1}  \bigg) \\
    & + (-1)^{n+1} \bigg\{ P_{x_1} (u_1) \cdot \alpha \big(  P_{x_2} (u_2) , \ldots, P_{x_{n+1}} (u_{n+1})  \big) + (-1)^{n+1} \alpha \big(  P_{x_1} (u_1), \ldots, P_{x_n} (u_n)  \big) \cdot P_{x_{n+1}} (u_{n+1}) \\
    & \qquad \qquad \qquad \qquad  + \sum_{i=1}^n (-1)^i \alpha \big(  P_{x_1} (u_1), \ldots, P_{x_i} (u_i) \cdot P_{x_{i+1}} (u_{i+1}), \ldots, P_{x_{n+1}} (u_{n+1})   \big) \bigg\} \\
    &- (-1)^{n+1} \sum_{r=1}^{n+1} P_{x_r}  \big(  (\delta^\alpha_\mathrm{Hoch} (\beta)) (P_{x_1} (u_1), \ldots, u_r, \ldots, P_{x_{n+1}} (u_{n+1}) )  \big).
\end{align*}
By using the fact that the collection $\{ P_x \}_{x \in X}$ of linear maps satisfy (\ref{matching-rel-rb-identity}), one gets that the above expression vanishes. Hence we have $ \delta_{ \{ P_x \}} \big( h_{\{ P_x \}} (\alpha, \beta) \big) ~+~  h_{\{ P_x \}} \big(  \delta_\mathrm{Hoch} (\alpha), \delta_\mathrm{Hoch}^\alpha (\beta) \big) = 0$. Therefore,
\begin{align*}
   & \big(  \delta_\mathrm{mrRBA} \big)^2 (\alpha, \beta, \gamma) \\
    &= \delta_\mathrm{mrRBA} \big(  \delta_\mathrm{Hoch} (\alpha),~ \delta_\mathrm{Hoch}^\alpha (\beta),~ \delta_{ \{ P_x \} }(\gamma) + h_{\{ P_x \}} (\alpha, \beta)  \big) \\
    &= \big( (\delta_\mathrm{Hoch})^2 (\alpha),~ (\delta_\mathrm{Hoch}^\alpha)^2 (\beta),~  (\delta_{ \{ P_x \} })^2(\gamma) +  \delta_{ \{ P_x \}} \big( h_{\{ P_x \}} (\alpha, \beta) \big) ~+~  h_{\{ P_x \}} (  \delta_\mathrm{Hoch} (\alpha), \delta_\mathrm{Hoch}^\alpha (\beta) ) \big) \\
    &= 0.
\end{align*}
Hence the result follows.
\end{proof}

Let
\begin{align*}
Z^n_{\mathrm{mrRBA}}(A,M,\{P_x\}_{x\in X}) \subseteq C^n_{\mathrm{mrRBA}}(A,M,\{P_x\}_{x\in X}),\\  B^n_{\mathrm{mrRBA}}(A,M,\{P_x\}_{x\in X}) \subseteq C^n_{\mathrm{mrRBA}}(A,M,\{P_x\}_{x\in X})
\end{align*}
be the space of $n$-cocycles and $n$-coboundaries, respectively. Then we have
\begin{align*}
B^n_{\mathrm{mrRBA}}(A,M,\{P_x\}_{x\in X}) \subseteq Z^{n}_{\mathrm{mrRBA}}(A,M,\{P_x\}_{x\in X}), \text{ for all } n \geq 0.      
\end{align*}
The corresponding quotients
\begin{center}
$H^n_{\mathrm{mrRBA}}(A,M,\{P_x\}_{x\in X}) := \cfrac{Z^n_{\mathrm{mrRBA}}(A,M,\{P_x\}_{x\in X})}{B^n_{\mathrm{mrRBA}}(A,M,\{P_x\}_{x\in X})}$ , for $n\geq 0$
\end{center}
are called the {\bf cohomology groups} of the matching relative Rota-Baxter algebra $(A,M,\{P_x\}_{x\in X})$.

\medskip

\subsection{Cohomology of matching Rota-Baxter algebras}
In this subsection, we define the cohomology of a matching Rota-Baxter algebra $(A,\{P_x\}_{x\in X})$ using the general framework of matching relative Rota-Baxter algebras.

Let $(A,\{P_x\}_{x\in X})$ be a matching Rota-Baxter algebra. For each $n\geq 0$, we define
\begin{align*}
C^n_{\mathrm{mRBA}}(A,\{P_x\}_{x\in X}) :=
\begin{cases}
    0,& \text{if } n= 0\\
    \mathrm{Hom}(A,A),              & \text{if } n = 1\\
    \mathrm{Hom}(A^{\otimes n},A)\oplus C^{n-1}_{\{P_x\}}(A,A),              & \text{if } n \geq 2.
\end{cases}
\end{align*}
Then there is an embedding $E:C^n_{\mathrm{mRBA}}(A,\{P_x\}_{x\in X}) \hookrightarrow C^n_{\mathrm{mrRBA}}(A,A,\{P_x\}_{x\in X})$ given by
\begin{align*}
E(\alpha,\gamma) = (\alpha,\alpha,\gamma), \text{ for } (\alpha,\gamma)\in C^n_{\mathrm{mRBA}}(A,\{P_x\}_{x\in X}).
\end{align*}
It is easy to verify that $\delta_{\mathrm{mrRBA}}(E(\alpha,\gamma))\subset \mathrm{im}(E)$.  This shows that the map $\delta_{\mathrm{mrRBA}}$ induces a map $\delta_{\mathrm{mRBA}}: C^n_{\mathrm{mRBA}}(A,\{P_x\}_{x\in X})\rightarrow C^{n+1}_{\mathrm{mRBA}}(A,\{P_x\}_{x\in X})$. The map $\delta_{\mathrm{mRBA}}$ is a differential as $\delta_{\mathrm{mrRBA}}$ is so.
Hence we obtain a cochain complex $\big\{C^{\bullet}_{\mathrm{mRBA}}(A,\{P_x\}_{x\in X}) , \delta_{\mathrm{mRBA}}\big\}$. The corresponding cohomology groups are called the {\bf cohomology} of the matching Rota-Baxter algebra $(A,\{P_x\}_{x\in X})$.

Let $(A,\{P_x\}_{x\in X})$ be a matching Rota-Baxter algebra.  Let $\big\{C^{\bullet}_{\{P_x\}}(A,A),\delta_{\{P_x\}}\big\}$ be the cochain complex induced by the collection of operators $\{P_x\}_{x\in X}$, and let $H^\bullet_{\{ P_x \}} (M, A)$ be the corresponding cohomology groups. Then there is a short exact sequence of cochain complexes
\begin{align}
0\rightarrow\big\{C^{\bullet - 1}_{\{P_x\}}(A,A) , \delta_{\{P_x\}}\big\} \xrightarrow{i} \big\{C^{\bullet}_{\mathrm{mRBA}}(A,\{P_x\}_{x\in X}),\delta_{\mathrm{mRBA}}\big\} \xrightarrow{p}\big\{C^{\bullet}_\mathrm{Hoch}(A,A),\delta_{\mathrm{Hoch}}\big\}\rightarrow 0,
\end{align}
where $i(f)=(0,f)$ and $p(\alpha,\gamma) = \alpha$, for $f\in  C^{\bullet - 1}_{\{P_x\}}(A,A)$ and $(\alpha,\gamma)\in C^\bullet_{\mathrm{rRBA}}(A,\{P_x\}_{x\in X})$. Thus, we obtain the following.
\begin{thm}
  Let $(A,\{P_x\}_{x\in X})$ be a matching Rota-Baxter algebra. Then there is a long exact sequence of cohomology groups
\begin{align*}
\cdots\rightarrow H^{n-1}_{\{P_x\}}(A,A)\rightarrow H^n_{\mathrm{mRBA}}(A,\{P_x\})\rightarrow H^n_{\mathrm{Hoch}}(A,A)\rightarrow H^{n}_{\{P_x\}}(A,A)\rightarrow\cdots.
\end{align*}
\end{thm}

\medskip

\section{Deformations of matching relative Rota-Baxter algebras}\label{sec5}
In this section, we study formal deformations of a matching relative Rota-Baxter algebra $(A,M,\{P_x\}_{x\in X})$. In such deformations, we allow to deform the underlying algebra $A$, the $A$-bimodule $M$ and the collection of operators $\{P_x\}_{x\in X}$. We show that such deformations are governed by the cohomology of the matching relative Rota-Baxter algebra $(A,M,\{P_x\}_{x\in X})$ introduced in the previous section. Finally, we consider infinitesimal deformations of a matching relative Rota-Baxter algebra as a truncated version of formal deformations.

Let $(A,M,\{P_x\}_{x\in X})$ be a matching relative Rota-Baxter algebra. We consider the space $A[[t]]$ (resp. $M[[t]]$) of formal power series in $t$ with coefficients in $A$ (resp. $M$). Then $A[[t]]$ and $M[[t]]$ are both $\mathbb{K}[[t]]$-modules.
%Moreover, the $A$-bimodule structure $M$ extends to an $A[[t]]$-bimodule structure on $\mathbb{K}[[t]]$.

\begin{defn} Let $(A,M,\{P_x\}_{x\in X})$ be a matching relative Rota-Baxter algebra. A {\bf formal one-parameter deformation} of $(A,M,\{P_x\}_{x\in X})$ consists of a tuple $(\mu_t,l_t,r_t, \{ (P_{t})_x \}_{x \in X})$ of three formal power series of the form
\begin{align*}
 \mu_t=\displaystyle\sum_{i\geq 0}t^i\mu_i, \quad l_t=\displaystyle\sum_{i\geq 0}t^i l_i, \quad r_t=\displaystyle\sum_{i\geq 0}t^i r_i
\end{align*}
and a collection of power series $\{ (P_{t})_x=\displaystyle\sum_{i\geq 0}t^i P_{i,x} \}_{x \in X}$ labelled by the elements of $X$ 
(where  $\mu_i\in \mathrm{Hom}(A^{\otimes{2}},A),~l_i\in \mathrm{Hom}(A \otimes M,A),~r_i\in \mathrm{Hom}(M \otimes A,A),~ P_{i, x} \in \mathrm{Hom}(M,A)$  for all $i\geq 0$, with $\mu_0=\mu,~l_0=~l,~r_0=r$ and $P_{0,x}=P_x$) such that

\medskip

$\bullet$  the $\mathbb{K}[[t]]$-linear operation $\mu_t$ makes $(A[[t]],\mu_t)$ into an associative algebra over $\mathbb{K}[[t]]$,

\medskip

$\bullet$ the $\mathbb{K}[[t]]$-linear operations $l_t,r_t$ makes $(M[[t]],l_t,r_t)$ into a bimodule over the associative algebra $(A[[t]],\mu_t)$,

\medskip

$\bullet$ the collection $\{(P_t)_x:M[[t]]\rightarrow A[[t]]\}_{x\in X}$ of $\mathbb{K}[[t]]$-linear maps satisfies
\begin{align*}
 \mu_t\big((P_t)_x (u),(P_t)_y (v) \big)=(P_t)_x\big(r_t(u,(P_t)_y v)\big)+(P_t)_y\big(l_t((P_t)_x u,v)\big), \text{ for all } u,v\in M \text{ and  } x,y\in X.
\end{align*}
In other words, the triple $\big(  (A[[t]],\mu_t),(M[[t]],l_t,r_t),\{(P_t)_x\}_{x\in X} \big)$ is a matching relative Rota-Baxter algebra over the ring $\mathbb{K}[[t]]$.
\end{defn}

It follows that the quadruple $(\mu_t,l_t,r_t, \{ (P_t)_x \}_{x \in X})$ is a formal one-parameter deformation if and only if the following system of equations are hold:
\begin{align}
\displaystyle\sum_{i+j=n} \mu_i(\mu_j(a,b),c) &=  \displaystyle\sum_{i+j=n} \mu_i(a,\mu_j(b,c)), \label{defor-eqn1}\\
\displaystyle\sum_{i+j=n} l_i(\mu_j(a,b),u) &= \displaystyle\sum_{i+j=n} l_i(a,l_j(b,u)),\\
\displaystyle\sum_{i+j=n} r_i(l_j(a,u),b) &= \displaystyle\sum_{i+j=n} l_i(a,r_j(u,b),\\
\displaystyle\sum_{i+j=n} r_i(r_j(u,a),b)&=\displaystyle\sum_{i+j=n} r_i(u,\mu_j(a,b)),\\
\displaystyle\sum_{i+j+k=n} \mu_i\big(P_{j,x} (u),P_{k,y}(v)\big)&=\displaystyle\sum_{i+j+k=n} P_{i,x} \big(r_j(u,P_{k,y}(v)\big) + \displaystyle\sum_{i+j+k=n} P_{i,y} \big(l_j(P_{k,x}(u),v)\big), \label{defor-eqn5}
\end{align}
for all $u,v\in M,~x,y\in X$ and $n\geq 0$. Note that all the above equations are trivially hold for $n=0$ as $(A,M,\{P_x\}_{x\in X})$ is a matching relative Rota-Baxter algebra. However, for $n=1$, we obtain
\begin{align}
\mu_1(a\cdot b,c) + \mu_1(a,b)\cdot c &= \mu_1(a,b\cdot c) + a\cdot\mu_1(b,c),\label{a}\\
l_1(a\cdot b,u) + \mu_1(a,b)\cdot_M u &= l_1(a,b\cdot_M u) + a\cdot_M l_1(b,u),\label{b1}\\
r_1(a\cdot_M u,b) + l_1(a,u)\cdot_M b &=  l_1(a,u\cdot_M b) + a\cdot_M r_1(u,b),\label{b2}\\
r_1(u\cdot_M a,b) + r_1(u,a)\cdot_M b &= r_1(u,a\cdot b) + u\cdot_M \mu_1(a,b),\label{b3}\\
P_{1,x}(u)\cdot P_y(v) + P_x(u)\cdot P_{1,y}(v) + \mu_1(P_x(u),P_y(v)) &= P_x \big(r_1(u,P_y(v)) + u\cdot_M P_{1,y} (v)\big) + P_{1,x} (u\cdot_MP_y(v))\nonumber\\
&~~+ P_y\big(l_1(P_x(u),v)+P_{1,x} (u)\cdot_Mv\big) + P_{1,y}(P_x(u)\cdot_Mv),\label{c}
\end{align}
for all $u,v\in M$ and $x,y\in X$. Note that the Equation (\ref{a}) is equivalent to $\delta_{\mathrm{Hoch}}(\mu_1) = 0.$ To better understand the remaining equations, we define an element $\beta_1\in \mathrm{Hom}(\mathcal{A}^{1,1},M)$ by
\begin{align*}
 \beta_1(a,u):=l_1(a,u) \quad \text{  and } \quad  \beta_1(u,a):=r_1(u,a), \text{ for } a\in A,u\in M.
\end{align*}
With this notation, the Equations (\ref{b1})-(\ref{b3}) can be simply expressed as $\delta^{\mu_1}_{\mathrm{Hoch}}(\beta_1)= 0.$ Finally, the equation (\ref{c}) is equivalent to
\begin{align*}
{\big(\delta_{\{P_x\}}(P_1) + h_{\{P_x\}}(\mu_1,\beta_1)\big)}_{x,y}(u,v) = 0 .  
\end{align*}
Thus, we have
\begin{align*}
\delta_{\mathrm{mrRBA}}(\mu_1,\beta_1,P_1) = \big(\delta_{\mathrm{Hoch}}(\mu_1),\delta^{\mu_1}_{\mathrm{Hoch}}(\beta_1),\delta_{\{P_x\}}(P_1)+h_{\{P_x\}}(\mu_1,\beta_1)\big)=0.
\end{align*}
This shows that $(\mu_1,\beta_1,P_1)\in Z^2_{\mathrm{mrRBA}}(A,M,\{P_x\}_{x\in X}) $ is a 2-cocycle, called the {\bf infinitesimal} of the given deformation. Hence it gives rise to a cohomology class in $H^2_{\mathrm{mrRBA}}(A,M,\{P_x\}_{x\in X})$.

\begin{defn} Let $(A,M,\{P_x\}_{x\in X})$ be a matching relative Rota-Baxter algebra. Two formal one-parameter deformations $(\mu_t,l_t,r_t,  \{ (P_t)_x \}_{x \in X})$ and $(\mu_t',l_t',r_t', \{ (P'_t)_x \}_{x \in X})$ are said to be {\bf equivalent} if there exist formal isomorphisms
\begin{align*}
\varphi_t = \displaystyle \sum_{i\geq 0} t^i \varphi_i: A[[t]]\rightarrow A[[t]] ~~ \text{  and } ~~ \psi_t = \displaystyle \sum_{i\geq 0} t^i\psi_i :M[[t]]\rightarrow M[[t]]
\end{align*}
(where $\varphi_i\in \mathrm{Hom}(A,A),\psi_i\in \mathrm{Hom}(M,M) $ for all $i\geq 0$, with $\varphi_0=\mathrm{id}_A$ and $\psi_0= \mathrm{id}_M$) such that the pair $(\varphi_t,\psi_t)$ is a morphism of matching relative Rota-Baxter algebras from
\begin{align*}
\big((A[[t]],\mu_t),(M[[t]],l_t,r_t),\{(P_t)_x\}_{x\in X}\big) ~ \text{ to } ~ \big((A[[t]],\mu_t'),(M[[t]],l_t',r_t'),\{(P_t')_x\}_{x\in X}\big).  
\end{align*}
\end{defn}

\medskip

Thus, it follows from the above definition that $(\mu_t,l_t,r_t,  \{ (P_t)_x \}_{x \in X})$ and $(\mu_t',l_t',r_t',  \{ (P'_t)_x \}_{x \in X})$ are equivalent if and only if the following system of equations are hold:
\begin{align}
\displaystyle\sum_{i+j=n} \varphi_i(\mu_j(a,b)) &=\displaystyle\sum_{i+j+k=n}\mu'_i(\varphi_j(a),\varphi_k(b)), \label{equiv-id1}\\
\displaystyle\sum_{i+j=n} \psi_i (l_j(a,u)) &=\displaystyle\sum_{i+j+k=n}l'_i(\varphi_j(a),\psi_j(u)),\\
\displaystyle\sum_{i+j=n} \psi_i (r_j(u,a)) &=\displaystyle\sum_{i+j+k=n}r'_i(\psi_j(u),\varphi_k(a)),\\
\displaystyle\sum_{i+j=n} \phi_i\circ P_{j,x}  &=\displaystyle\sum_{i+j=n} P'_{i,x} \circ \psi_j, \label{equiv-id4}
\end{align}
for all $a,b\in A,~u\in M,~x\in X$ and $n\geq 0$. All these relations are hold for $n=0$ as $\varphi_0=\mathrm{id}_A$ and $\psi_0=\mathrm{id}_M$. However, for $n=1$, we obtain four relations which are equivalent to the single identity
\begin{align*}
(\mu_1,\beta_1,P_1) - (\mu_1',\beta_1',P_1')  =\delta_{\mathrm{mrRBA}} (\varphi_1,\psi_1).
\end{align*}
This shows that the 2-cocycles $(\mu_1,\beta_1,P_1)$ and $(\mu_1',\beta_1',P_1')$ are cohomologous. Hence they corresponds to the same cohomology class. As a summary of the above discussions, we get the following.

\begin{thm} Let $(A,M,\{P_x\}_{x\in X})$ be a matching relative Rota-Baxter algebra. Then the infinitesimal of any formal one-parameter deformation of $(A,M,\{P_x\}_{x\in X})$ is a 2-cocycle. Moreover, the corresponding cohomology class in $H^2_{\mathrm{mrRBA}}(A,M,\{P_x\}_{x\in X})$  depends only on the equivalence class of the deformation.
\end{thm}

In the following, we will consider a truncated version of formal deformations. Such deformations are called infinitesimal deformations. More precisely, we have the following definition.

\begin{defn} Let $(A,M,\{P_x\}_{x\in X})$ be a matching relative Rota-Baxter algebra. An {\bf infinitesimal deformation} of $(A,M,\{P_x\}_{x\in X})$ is a formal one-parameter deformation over the ring $\mathbb{K}[[t]]/(t^2)$.
\end{defn}

Thus, an infinitesimal deformation $(\mu_t,l_t,r_t,  \{ (P_t)_x \}_{x \in X})$ is given by
\begin{align*}
 \mu_t =\mu + t\mu_1 , \quad l_t= l +t l_1, \quad r_t=r+t r_1 ~~ \text{ and } ~~ \{ (P_t)_x=P_x +t P_{1,x} \}_{x \in X}   
\end{align*}
such that the identities (\ref{defor-eqn1})-(\ref{defor-eqn5}) holds for $n=0,1$.
Similarly, one can define equivalence between two infinitesimal deformations. More precisely, let $(\mu_t,l_t,r_t,  \{ (P_t)_x \}_{x \in X})$ and $(\mu_t',l_t',r_t',  \{ (P'_t)_x \}_{x \in X})$ be two infinitesimal deformations. They are equivalent if there exist sums
\begin{align*}
\varphi_t= \mathrm{id}_A+t\varphi_1 ~~ \text{ and  } ~~ \psi_t= \mathrm{id}_M+t\psi_1
\end{align*}
such that the identities (\ref{equiv-id1})-(\ref{equiv-id4}) are hold for $n=0,1$. We are now ready to give another important result of this section.

\begin{thm}
 Let $(A,M,\{P_x\}_{x\in X})$ be a matching relative Rota-Baxter algebra. The set of all equivalence classes of infinitesimal deformations of $(A,M,\{P_x\}_{x\in X})$ is classified by the second cohomology group $H^2_{\mathrm{mrRBA}}(A,M,\{P_x\}_{x\in X})$.
 \end{thm}

\begin{proof} One can easily show that any infinitesimal deformation gives rise to a $2$-cocycle and the corresponding cohomology class depends only on the equivalence class of the infinitesimal deformation.

Conversely, given a 2-cocycle $(\mu_1,\beta_1,P_1)\in Z^2_{\mathrm{mrRBA}}(A,M,\{P_x\}_{x\in X})$, one can construct an infinitesimal deformation  $(\mu_t,l_t,r_t,  \{ (P_t)_x \}_{x \in X})$ as follows:
\begin{align*}
 \mu_t=\mu+t\mu_1 , \quad l_t=l+t l_1, \quad r_t=r+t r_1,  \quad \mathrm{and} \quad  \{ (P_t)x =P_x +t (P_1)_x \}_{x \in X},  
\end{align*}
where $l_1(a,u):=\beta_1(a,u)$ and $r_1(u,a):=\beta_1(u,a)$, for all $a\in A,~u\in M$. Moreover, it is easy to see that if $(\mu_1,\beta_1,P_1)$ and $(\mu'_1,\beta'_1,P_1')$ are two cohomologous $2$-cocycles then the corresponding infinitesimal deformations are equivalent. This proves the desired result.
\end{proof}

\medskip

\section{Matching dendriform algebras}\label{sec6}
In this section, we introduce the cohomology of a matching dendriform algebra using nonsymmetric operads. As an application of the cohomology, we study deformations of a matching dendriform algebra. Finally, we discuss a relationship between the cohomology of a matching relative Rota-Baxter algebra and the cohomology of the induced matching dendriform algebra. We first recall some basics about nonsymmetric operads.

\begin{defn} \cite{gers-voro,lod-val-book} (i) A {\bf nonsymmetric operad} is a pair $\mathcal{O}=\big(\{\mathcal{O}(k)\}_{k=1}^{\infty}, \circ_i\big)$ consisting of a collection $\{\mathcal{O}(k)\}_{k=1}^{\infty}$ of $\mathbb{K}$-modules and linear maps (called partial compositions)
\begin{align*}
 \circ_i: \mathcal{O}(k)\otimes \mathcal{O}(l)\rightarrow\mathcal{O}(k+l-1), \text{ for } k,l\geq 1 \text{ and } 1\leq i\leq k
\end{align*}
satisfying
\begin{align*}
(f \circ_i g) \circ_{i+j-1}h =~& f \circ_i (g \circ_j h), \text{ for } 1\leq j\leq k,~ 1\leq j\leq l,\\
(f \circ_i g) \circ_{j+n-1} h =~& (f \circ_j h) \circ_{i}g, \text{ for }  1\leq i<j\leq k,
\end{align*}
for $f\in \mathcal{O}(k),~g\in \mathcal{O}(l),~h\in \mathcal{O}(m)$; and there exists an element $\mathbbm{1} \in \mathcal{O}(1)$ that satisfies $f \circ_i \mathbbm{1} = f = \mathbbm{1} \circ_1 f$, for $f\in\mathcal{O}(k)$ and $1\leq i\leq k.$

\medskip

(ii) A {\bf multiplication} in a nonsymmetric operad $\mathcal{O}=\big(\{\mathcal{O}(k)\}_{k=1}^{\infty},\circ_i\big)$ is an element $\pi\in \mathcal{O}(2)$ that satisfies $\pi \circ_1 \pi = \pi \circ_2 \pi$.
\end{defn}

Let $\mathcal{O}=\big(\{\mathcal{O}(k)\}_{k=1}^{\infty}, \circ_i\big)$ be a nonsymmetric operad. Then the graded $\mathbb{K}$-module $\displaystyle \oplus_{k \geq 1}\mathcal{O}(k)$ inherits a degree $-1$ graded Lie bracket given by
\begin{align*}
\{ \! \! \{ f,g \} \! \! \}:= \displaystyle\sum_{i=1}^{k}(-1)^{(i-1)(l-1)} f \circ_i g~-~(-1)^{(k-1)(l-1)} \displaystyle\sum_{i=1}^{l} (-1)^{(i-1)(k-1)}g \circ_i f,  
\end{align*}
for $f\in \mathcal{O}(k)$ and $g\in \mathcal{O}(l)$. In other words, the shifted graded $\mathbb{K}$-module $\displaystyle \oplus_{k\geq 0}\mathcal{O}(k+1)$ with the above bracket forms a graded Lie algebra. Additionally, let $\pi\in \mathcal{O}(2)$ be a multiplication on the operad $\mathcal{O}$. Then $\pi$ induces a differential
\begin{align*}
\delta_{\pi}:\mathcal{O}(n)\rightarrow \mathcal{O}(n+1),~ \delta_{\pi}(f)=(-1)^{n+1} \{ \! \! \{ \pi,f \} \! \! \},~ \text{ for } f\in \mathcal{O}(n).
\end{align*}
Therefore, $\{\mathcal{O}(\bullet),\delta_{\pi}\}$ becomes a cochain complex. The corresponding cohomology groups are called the cohomology induced by the multiplication $\pi$.

Next, we aim to construct a new nonsymmetric operad associated to a $\mathbb{K}$-module $D$. First, let $C_k$ be the set of first $k$ natural numbers. For our convenience, we write $C_k = \{ [1], [2], \ldots, [k] \}$. For any $k\geq 1$,  let $\mathfrak{h}^k$ be the set of all elements of the form $f = \displaystyle \oplus_{{x_1,\ldots,x_k}\in X} f_{x_1,\ldots,x_k}$ with $f_{x_1,\ldots,x_k}\in \mathrm{Hom}\big(\mathbb{K}[C_k]\otimes D^{\otimes{k}},D \big)$. Thus, an element $f\in \mathfrak{h}^k$ is given by a collection $\big\{f_{x_1,\ldots,x_k}:\mathbb{K}[C_k]\otimes D^{\otimes{k}}\rightarrow D\big\}_{{x_1,\ldots,x_k}\in X}$ of maps labelled by the elements of $X^{\times k}$. For $[r] \in C_k$, we denote the map $f_{x_1, \ldots, x_k} \big( [r]; ~, \ldots, ~ ) : D^{\otimes k} \rightarrow D$ simply by $f^{[r]}_{x_1, \ldots, x_k}$. Note that $\mathfrak{h}^k$ is a $\mathbb{K}$-madule with the obvious abelian group structure and the $\mathbb{K}$-action. We define
\begin{align*}
\mathcal{O}_D(k)=\big\{f = \displaystyle \oplus_{{x_1,\ldots,x_k}\in X} f_{x_1,\ldots,x_k}\in \mathfrak{h}^k ~~|~~ f_{x_1,\ldots,x_k}^{[r]} \text{ doesn't depend on } x_r, \text{ for } 1\leq r\leq k \big\}.  
\end{align*}
For any $f\in \mathcal{O}_D(k),~g\in\mathcal{O}_D(l)$ and $1\leq i\leq k$, we define an element $f \circ_i g \in \mathcal{O}_D(k+l-1)$ by

\medskip
\begin{align}\label{circ-dend-match}
&(f\circ_ig)^{[r]}_{x_1,\ldots,x_{k+l-1}}(a_1,\ldots,a_{k+l-1}) =\\
&= \begin{cases}
\sum_{s=1}^l f^{[r]}_{x_1,\ldots,x_{i-1}, x_{i-1+s},x_{i+l},\ldots,x_{k+l-1}} \big(a_1,\ldots,a_{i-1},g^{[s]}_{x_i,\ldots,x_{i+l-1}}(a_i,\ldots,a_{i+l-1}),a_{i+l},\ldots,a_{k+l-1}\big),\\
 \hspace{12cm}   \text{ if } 1 \leq r \leq i-1\medskip \medskip \\
 f^{[i]}_{x_1,\ldots,x_{i-1},x_r,x_{i+l},\ldots,x_{k+l-1}} \big(a_1,\ldots,a_{i-1},g^{[r-i+1]}_{x_i,\ldots,x_{i+l-1}}(a_i,\ldots,a_{i+l-1}),\ldots,a_{k+l-1}\big), ~ \text{ if } i \leq r \leq i+l-1 \\\\ \medskip \medskip
\sum_{s=1}^l f^{[r-l+1]}_{x_1,\ldots,x_{i-1},x_{i-1+s},x_{i+l},\ldots,x_{k+l-1}} \big(a_1,\ldots,a_{i-1},g^{[s]}_{x_i,\ldots,x_{i+l-1}}(a_i,\ldots,a_{i+l-1}),a_{i+l},\ldots,a_{k+l-1}\big),\\
\hspace{10cm} \text{ if }  i+l \leq r \leq k+l-1, \nonumber
\end{cases}
\end{align}
for $[r] \in C_{k+l-1}$, $x_1, \ldots, x_{k+l-1} \in X$ and $a_1, \ldots, a_{k+l-1} \in D$. With the above notations, it is easy to see that the pair $\mathcal{O}_D=\big(\{\mathcal{O}_D(k)\}_{k=1}^{\infty}, \circ_i\big)$ is a nonsymmetric operad. This operad generalizes the one constructed in \cite{das-dend} to study ordinary dendriform algebras.

\begin{thm} Let $D$ be a $\mathbb{K}$-module. A matching dendriform algebra structure on $D$ is equivalent to a multiplication on the nonsymmetric operad $\mathcal{O}_D$.
\end{thm}

\begin{proof}
First, note that an element $\pi = \oplus_{x, y \in X} \pi_{x,y} \in \mathcal{O}_D (2)$ is equivalent to a collection of linear maps $\{ \prec_x, \succ_x : D \otimes D \rightarrow D \}_{x \in X}$. The correspondence is given by
\begin{align*}
\pi^{[1]}_{x,y} = a \prec_y b \quad \text{ and } \quad \pi^{[2]}_{x,y} = a \succ_x b, \text{ for } x, y \in X \text{ and } a, b \in A.
\end{align*}
Then it follows from (\ref{circ-dend-match}) that
\begin{align*}
\big( \pi \circ_1 \pi - \pi \circ_2 \pi  \big)^{[1]}_{x,y,z} (a,b,c) =~& (a \prec_y b) \prec_z c - a \prec_y (b \prec_z c) - a \prec_z (b \succ_y c) ,\\
\big( \pi \circ_1 \pi - \pi \circ_2 \pi  \big)^{[2]}_{x,y,z} (a,b,c) =~& (a \succ_x b) \prec_z c - a \succ_x (b \prec_z c),\\
\big( \pi \circ_1 \pi - \pi \circ_2 \pi  \big)^{[3]}_{x,y,z} (a,b,c) =~& (a \prec_y b) \succ_x c + (a \succ_x b) \succ_y c - a \succ_x (b \succ_y c),
\end{align*}
for $x,y,z \in A$ and $a, b, c \in A$. This shows that $\pi \circ_1 \pi = \pi \circ_2 \pi$ (i.e. $\pi$ is a multiplication in the operad $\mathcal{O}_D$) if and only if the collection $\{ \prec_x , \succ_x \}_{x \in X}$ satisfies the matching dendriform algebra identities (\ref{mda-1})-(\ref{mda-3}). This completes the proof.
\end{proof}

Let $(D,\{\prec_x,\succ_x\}_{x\in X})$ be a given matching dendriform algebra. Then by the previous theorem, this matching dendriform algebra structure induces the multiplication $\pi_D \in \mathcal{O}_D (2)$ on the operad $\mathcal{O}_D$. For each $n\geq 1$, we define
\begin{align*}
C^n_{\mathrm{mDA}}(D,\{\prec_x,\succ_x\}_{x\in X})=\mathcal{O}_D(n)=  \big\{f = \displaystyle \oplus_{{x_1,\ldots,x_n}\in X} f_{x_1,\ldots,x_n} ~~|~~ f_{x_1,\ldots,x_n}^{[r]} & \text{ doesn't depend on } x_r,\\ & \text{ for }  1\leq r\leq n \big\}  
\end{align*}
and a map $\delta_{\mathrm{mDA}}:C^n_{\mathrm{mDA}}(D,\{\prec_x,\succ_x\}_{x\in X})\rightarrow C^{n+1}_{\mathrm{mDA}}(D,\{\prec_x,\succ_x\}_{x\in X})$ by
\begin{align*}
\delta_{\mathrm{mDA}}(f)=~& (-1)^{n+1} \{ \! \! \{ \pi_D,f \} \! \! \}\\
=~&(-1)^{n+1}\big(\pi_D \circ_1 f+(-1)^{n-1}~\pi_D \circ_2 f - \displaystyle\sum_{i=1}^{n}(-1)^{i-1} f \circ_i \pi_D \big)\\
=~& \pi_D \circ_2f + \displaystyle\sum_{i=1}^{n} (-1)^i~ f \circ_i \pi_D + (-1)^{n+1}~\pi_D \circ_1 f.
\end{align*}
Since $\pi_D$ is a multiplication, it follows that $\big\{C^\bullet_{\mathrm{mDA}}(D,\{\prec_x,\succ_x\}_{x\in X}),~\delta_{\mathrm{mDA}}\big\}$ is a cochain complex. The corresponding cohomology groups are called the {\bf cohomology} of the given matching dendriform algebra $(D,\{\prec_x,\succ_x\}_{x\in X})$, and they are denoted by $H^\bullet_\mathrm{mDA} (D,\{\prec_x,\succ_x\}_{x\in X})$.

\begin{remark}
In Remark \ref{embedding} we have seen that a matching dendriform algebra $(D, \{ \prec_x, \succ_x \}_{x \in X})$ can be embedded into the matching Rota-Baxter algebra $\big( (D \otimes \mathbb{K}[X]) \oplus D, \{ P_x \}_{x \in X} \big)$. For each $n \geq 1$, we define a map $\mathcal{H}_n : C^n_\mathrm{mDA} ( D, \{ \prec_x, \succ_x \}_{x \in X}) \rightarrow C^n_\mathrm{Hoch} \big(  (D \otimes \mathbb{K}[X]) \oplus D, (D \otimes \mathbb{K}[X]) \oplus D \big)$ by
\begin{align*}
    \mathcal{H}_n (f) \big(  (a_1 \otimes x_1, 0), \ldots, (a_n \otimes x_n, 0)  \big) =~& \big(  \sum_{r=1}^n ( f^{[r]}_{x_1, \ldots, x_n} (a_1, \ldots, a_n) \otimes x_r), 0  \big),\\
    \mathcal{H}_n (f)  \big(  (a_1 \otimes x_1, 0), \ldots, (0, a_i'), \ldots,  (a_n \otimes x_n, 0)  \big) =~& \big( 0, f^{[i]}_{x_1, \ldots, x_n} (a_1, \ldots,a_i', \ldots, a_n)  \big),\\
    \text{ and the value of } \mathcal{H}_n(f) \text{ is zero if two or more} & \text{ inputs are of the form } (0, a).
\end{align*}
Then it is straightforward to verify that $\delta_\mathrm{Hoch} \circ \mathcal{H}_n  = \mathcal{H}_{n+1} \circ \delta_\mathrm{mDA}$, for all $n \geq 1$. Hence there is a morphism $\mathcal{H}_\bullet : H^\bullet_\mathrm{mDA} (  D, \{ \prec_x, \succ_x \}_{x \in X}   ) \rightarrow H^\bullet_\mathrm{Hoch}   \big(   (D \otimes \mathbb{K}[X]) \oplus D, (D \otimes \mathbb{K}[X]) \oplus D \big)$ between cohomology groups.
\end{remark}

\subsection{Deformations of matching dendriform algebras} In this subsection, we give an application of the cohomology of a matching dendriform algebra. More precisely, we study formal one-parameter deformations of a matching dendriform algebra and show that such deformations are closely related to the cohomology theory.

\begin{defn}
(i) Let $(D, \{ \prec_x, \succ_x \}_{x \in X})$ be a matching dendriform algebra. A {\bf formal one-parameter deformation} of $(D, \{ \prec_x, \succ_x \}_{x \in X})$ consists of a collection $\{ (\prec_t)_x, (\succ_t)_x \}_{x \in X}$ of formal power series
\begin{align*}
(\prec_t)_x = \sum_{i \geq 0} t^i \prec_{i, x} \quad \text{ and } \quad (\succ_t)_x = \sum_{i \geq 0} t^i \succ_{i, x}
\end{align*}
(where $\prec_{i,x} , \succ_{i,x} \in \mathrm{Hom} (D^{\otimes 2}, D)$ for each $i \geq 0$, with $\prec_{0,x} = \prec_x$ and $\succ_{0,x} = \succ_x$ for all $x \in X$) such that the pair $(D[[t]], \{ (\prec_t)_x, (\succ_t)_x \}_{x \in X})$ becomes a matching dendriform algebra over $\mathbb{K}[[t]]$.

\medskip

(ii) Two formal one-parameter deformations $\{ (\prec_t)_x, (\succ_t)_x \}_{x \in X}$ and $\{ (\prec'_t)_x, (\succ'_t)_x \}_{x \in X}$ are said to be {\bf equivalent} if there exists a formal isomorphism $\psi_t = \sum_{i \geq 0} t^i \psi_i : D[[t]] \rightarrow D[[t]]$ of $\mathbb{K}[[t]]$-modules (where $\psi_i \in \mathrm{Hom}(D,D)$ for all $i \geq 0$, with $\psi_0 = \mathrm{id}_D$) such that $\psi_t$ is a morphism of matching dendriform algebras from $(D[[t]], \{ (\prec_t)_x, (\succ_t)_x \}_{x \in X})$ to $(D[[t]], \{ (\prec'_t)_x, (\succ'_t)_x \}_{x \in X})$.
\end{defn}

It follows from the above definition that $\{ (\prec_t)_x, (\succ_t)_x \}_{x \in X}$ is a formal one-parameter deformation if and only if the following system of equations (called deformation equations) are hold
\begin{align}
\sum_{i+j = n} (a\prec_{i,x} b)\prec_{j,y} c&~= \sum_{i+j = n} a\prec_{i,x} (b\prec_{j,y} c) + a\prec_{i,y} (b\succ_{j,x} c), \label{mda-def-1}\\
\sum_{i+j = n} (a\succ_{i,x} b)\prec_{j,y} c &~= \sum_{i+j = n}a\succ_{i,x} (b\prec_{j,y} c), \label{mda-def-2}\\
\sum_{i+j = n} (a\prec_{i,y} b)\succ_{j,x} c + (a\succ_{i,x} b)\succ_{j,y} c &~= \sum_{i+j = n} a\succ_{i,x} (b\succ_{j,y} c), \label{mda-def-3}
\end{align}
for all $a,b,c\in D$, $x,y\in X$ and $n \geq 0$. These equations are obviously hold for $n=0$ as $\prec_{0,x} = \prec_x$ and $\succ_{0,x} = \succ_x$, for all $x \in X$. To write all the above identities for $n=1$ in a compact form, we consider an element $\pi_1 = \oplus_{x, y \in X} (\pi_1)_{x,y} \in \mathcal{O}_D(2) = C^2_\mathrm{mDA} (D, \{ \prec_x, \succ_x \}_{x \in X})$ defined by
\begin{align*}
(\pi_1)^{[1]}_{x,y} (a, b) = a \prec_y b \quad \text{ and } \quad (\pi_1)^{[2]}_{x,y} (a, b) = a \succ_x b, \text{ for } a, b \in A \text{ and } x, y \in X.
\end{align*}
Then it is straightforward to verify that the identities (\ref{mda-def-1})-(\ref{mda-def-3}) for $n=1$ can be equivalently written as $\delta_\mathrm{mDA} (\pi_1) = 0$. In other words, $\pi_1$ is a $2$-cocycle, called the infinitesimal of the given formal one-parameter deformation $\{ (\prec_t)_x, (\succ_t)_x \}_{x \in X}$.

On the other hand, two formal one-parameter deformations $\{ (\prec_t)_x, (\succ_t)_x \}_{x \in X}$ and $\{ (\prec'_t)_x, (\succ'_t)_x \}_{x \in X}$ are equivalent if and only if the following system of equations are hold:
\begin{align*}
\sum_{i+j = n} \psi_{i} \big(  a \prec_{j,x} b \big) = \sum_{i+j+k = n} \psi_i (a) \prec'_{j, x} \psi_k (b) ~~~ \text{ and } ~~~ \sum_{i+j = n} \psi_{i} \big(  a \succ_{j,x} b \big) = \sum_{i+j+k = n} \psi_i (a) \succ'_{j, x} \psi_k (b),
\end{align*}
for $a, b \in A$, $x \in X$ and $n \geq 0$. For $n=1$, both of these equations can be combinedly written as
\begin{align}\label{difference-coboundary}
\pi_1 - \pi_1' = \delta_\mathrm{mDA} (\psi_1).
\end{align}
Here we consider $\psi_1 \in \mathrm{Hom}(D,D)$ as an element of $C^1_\mathrm{mDA} (   D , \{ \prec_x, \succ_x \}_{x \in X})$ by $(\psi_1)^{[1]}_x (a) = \psi_1 (a)$, for all $a \in A$ and $x \in X$. The equation (\ref{difference-coboundary}) shows that the infinitesimal corresponding to equivalent deformations are cohomologous. Hence we obtain the following result.

\begin{thm}
Let $(D, \{ \prec_x, \succ_x \}_{x \in X})$ be a matching dendriform algebra. Then the infinitesimal of any formal one-parameter deformation of $(D, \{ \prec_x, \succ_x \}_{x \in X})$ is a $2$-cocycle in the cohomology complex of the matching dendriform algebra. Moreover, the corresponding class in $H^2_\mathrm{mDA} (D, \{ \prec_x, \succ_x \}_{x \in X})$ depends on the equivalence class of the formal one-parameter deformation.
\end{thm}

One could also define infinitesimal deformations of a matching dendriform algebra as a formal deformation over the ring $\mathbb{K}[[t]]/(t^2)$. Then it can be shown that the set of all equivalence classes of infinitesimal deformations of a matching dendriform algebra $(D, \{ \prec_x, \succ_x \}_{x \in X})$ is classified by the cohomology group $H^2_\mathrm{mDA} (D, \{ \prec_x, \succ_x \}_{x \in X}).$
\subsection{A relation with the cohomology of matching relative Rota-Baxter algebras}
In section \ref{sec3}, we have seen that a matching relative Rota-Baxter algebra induces a matching dendriform algebra. In this subsection, we find a relation between the cohomology of a matching relative Rota-Baxter algebra and the cohomology of the induced matching dendriform algebra.

Let $(A,M,\{P_x\}_{x\in X})$ be a matching relative Rota-Baxter algebra. We consider the induced matching dendriform algebra $(M,\{\prec_x,\succ_x\}_{x\in X})$ given by
\begin{align*}
u\prec_xv=u\cdot_MP_x(v) ~\text{ and }~ u\succ_xv=P_x(u)\cdot_Mv, \text{ for } u,v\in M \text{ and } x\in X.
\end{align*}
Let $\pi_M\in {\mathcal{O}_M(2)}$ be the corresponding multiplication on the operad $\mathcal{O}_M$. Then we have the following result.

\begin{prop} The collection $\{\theta_n\}_{n=1}^{\infty}$ of maps $\theta_n:C^n_{\{P_x\}}(M,A)\rightarrow C^{n+1}_{\mathrm{mDA}}\big(M,\{\prec_x,\succ_x\}_{x\in X}\big)$ given by
\[
    (\theta_n(f))^{[r]}_{x_1,\ldots,x_{n+1}}(u_1,\ldots,u_{n+1})=
\begin{cases}
    (-1)^{n+1}~~u_1 \cdot_M f_{x_2,\ldots,x_{n+1}}(u_2,\ldots,u_{n+1}), & \mathrm{ if }~ r= 1\\
    0,              & \mathrm{ if }~  2\leq r \leq n\\
    f_{x_1,\ldots,x_n}(u_1,\ldots,u_n)\cdot_Mu_{n+1},              & \mathrm{ if } ~r = n+1
\end{cases}
\]
is a graded Lie algebra morphism from $\big(\oplus_{n\geq 0}C^{n}_{\{P_x\}}(M,A),\llbracket~,~\rrbracket \big)$ to $\big(\oplus _{n\geq 0}C^{n+1}_{\mathrm{mDA}}(M,\{\prec_x,\succ_x\}_{x\in X}),\{ \! \! \{ , \} \! \! \} \big)$.
\\

(ii) The collection $\{\theta_n\}_{n=1}^{\infty}$ induces a morphism $\theta_{\bullet}: H^{\bullet}_{\{P_x\}}(M,A)\rightarrow H^{\bullet+1}_{\mathrm{mDA}}(M,\{\prec_x,\succ_x\}_{x\in X})$ between the cohomology groups.
\end{prop}
\begin{proof}
(i) Let $f \in C^m_{\{P_x\}}(M,A)$ and $g \in C^n_{\{P_x\}}(M,A)$. Then by a direct calculation, we have
\begin{align*}
    &\big(  \{ \! \! \{ \theta_m (f) , \theta_n (g)  \} \! \! \} \big)^{[1]}_{x_1, \ldots, x_{m+n+1}} (u_1, \ldots, u_{m+n+1} ) \\
    &= \big(  \sum_{i=1}^{m+1} (-1)^{(i-1)n}~ \theta_m(f) \circ_i \theta_n (g) -(-1)^{mn} \sum_{i=1}^{n+1} (-1)^{(i-1)m}~ \theta_n(g) \circ_i \theta_m (f)   \big)^{[1]}_{x_1, \ldots, x_{m+n+1}} (u_1, \ldots, u_{m+n+1} )   \\
    &= (\theta_m (f))^{[1]}_{x_1, x_{n+2}, \ldots, x_{m+n+1}} \big( (\theta_n (g))^{[1]}_{x_1, \ldots, x_{n+1}} (u_1, \ldots, u_{n+1}), u_{n+2}, \ldots, u_{m+n+1}  \big) ~+ \\
    & \sum_{i=1}^m (-1)^{in} \sum_{s=1}^{n+1} (\theta_m (f))^{[1]}_{x_1, \ldots, x_i, x_{i+s}, x_{i+n+2}, \ldots, x_{m+n+1}} \big( u_1, \ldots, u_i, (\theta_n (g))^{[s]}_{x_{i+1}, \ldots, x_{i+n+1}} (u_{i+1}, \ldots, u_{i+n+1}), \ldots \big)\\
   & - (-1)^{mn} \bigg\{    
    (\theta_n (g))^{[1]}_{x_1, x_{m+2}, \ldots, x_{m+n+1}} \big( (\theta_m (f))^{[1]}_{x_1, \ldots, x_{m+1}} (u_1, \ldots, u_{m+1}), u_{m+2}, \ldots, u_{m+n+1}  \big) ~+ \\
    & \small{ \sum_{i=1}^n (-1)^{im} \sum_{s=1}^{m+1} (\theta_n (g))^{[1]}_{x_1, \ldots, x_i, x_{i+s}, x_{i+m+2}, \ldots, x_{m+n+1}} \big( u_1, \ldots, u_i, (\theta_m (f))^{[s]}_{x_{i+1}, \ldots, x_{i+m+1}} (u_{i+1}, \ldots, u_{i+m+1}), \ldots \big)}
    \bigg\}\\
   & = (-1)^{m+n+1} u_1 \cdot_M \big( \llbracket f, g \rrbracket_{x_2, \ldots, x_{m+n+1}} (u_2, \ldots, u_{m+n+1})  \big)  = \big( \theta_{m+n} \llbracket f, g \rrbracket  \big)^{[1]}_{x_1, \ldots, x_{m+n+1}} (u_1, \ldots, u_{m+n+1}).
\end{align*}
On the other hand, if $2 \leq r \leq m+n$, by a simple observation we get that
\begin{align*}
    \big(  \{ \! \! \{ \theta_m (f) , \theta_n (g)  \} \! \! \} \big)^{[r]}_{x_1, \ldots, x_{m+n+1}} (u_1, \ldots, u_{m+n+1} ) = \big( \theta_{m+n} \llbracket f, g \rrbracket  \big)^{[r]}_{x_1, \ldots, x_{m+n+1}} (u_1, \ldots, u_{m+n+1}) = 0.
\end{align*}
Finally,
\begin{align*}
   & \big(  \{ \! \! \{ \theta_m (f) , \theta_n (g)  \} \! \! \} \big)^{[m+n+1]}_{x_1, \ldots, x_{m+n+1}} (u_1, \ldots, u_{m+n+1} ) \\
    &= \big(  \sum_{i=1}^{m+1} (-1)^{(i-1)n}~ \theta_m(f) \circ_i \theta_n (g) -(-1)^{mn} \sum_{i=1}^{n+1} (-1)^{(i-1)m}~ \theta_n(g) \circ_i \theta_m (f)   \big)^{[m+n+1]}_{x_1, \ldots, x_{m+n+1}} (u_1, \ldots, u_{m+n+1} )   \\
   & = \sum_{i=1}^m (-1)^{(i-1)n} \sum_{i=1}^{n+1} (\theta_m (f))^{[m+1]}_{x_1, \ldots, x_{i-1}, x_{i-1+s}, x_{i+n+1}, \ldots, x_{m+n+1}} \big(  u_1, \ldots, u_{i-1}, (\theta_n (g))^{[s]}_{x_i, \ldots, x_{i+n}} (u_i, \ldots, u_{i+n} ), \ldots \big) \\
   &~ + (-1)^{mn} (\theta_m (f))^{[m+1]}_{x_1, \ldots, x_m, x_{m+n+1}} \big( u_1, \ldots, u_m , (\theta_n (g))^{[n+1]}_{x_{m+1}, \ldots, x_{m+n+1}} ( u_{m+1}, \ldots, u_{m+n+1}  ) \big) - (-1)^{mn}   \\
   &~ \bigg\{  \sum_{i=1}^n (-1)^{(i-1)m} \sum_{i=1}^{m+1} (\theta_n (g))^{[n+1]}_{x_1, \ldots, x_{i-1}, x_{i-1+s}, x_{i+m+1}, \ldots, x_{m+n+1}} \big(  u_1, \ldots, u_{i-1}, (\theta_m (f))^{[s]}_{x_i, \ldots, x_{i+m}} (u_i, \ldots, u_{i+m} ), \ldots   \big) \\
  &~ + (-1)^{mn} (\theta_n (g))^{[n+1]}_{x_1, \ldots, x_n, x_{m+n+1}} \big( u_1, \ldots, u_n , (\theta_m (g))^{[m+1]}_{x_{n+1}, \ldots, x_{m+n+1}} ( u_{n+1}, \ldots, u_{m+n+1}  ) \big)
    \bigg\} \\
  &  = \big(  \llbracket f, g \rrbracket_{x_1, \ldots, x_{m+n}} (u_1, \ldots, u_{m+n}) \big) \cdot_M u_{m+n+1} = \big( \theta_{m+n}  \llbracket f, g \rrbracket  \big)^{[m+n+1]}_{x_1, \ldots, x_{m+n+1}} (u_1, \ldots, u_{m+n+1}).
\end{align*}
This completes the proof.

\medskip

(ii) Let $P = \oplus_{x \in X} P_x$ be the Maurer-Cartan element in the graded Lie algebra $(\oplus_{n \geq 0} C^n_{ \{ P_x \} } (M,A), \llbracket ~, ~\rrbracket).$ Then we have
\begin{align*}
    (\theta_1 (P))^{[1]}_{x, y} (u,v) = u \cdot_M P_y (v) ~~~ \text{ and } ~~~ (\theta_1 (P))^{[2]}_{x, y} (u,v) = P_x (u) \cdot_M v, \text{ for } u, v \in M \text{ and } x, y \in X. 
\end{align*}
This shows that $\theta_1 (P) \in \mathcal{O}_M (2)$ is the multiplication in the operad $( \{ \mathcal{O}_M (k) \}_{k=1}^\infty , \circ_i ) $ corresponding to the induced matching dendriform algebra $(M, \{ \prec_x, \succ_x \}_{x \in X})$. Hence for any $f \in C^n_{ \{ P_x \} } (M,A)$, we have
\begin{align*}
    (\theta_{n+1} \circ \delta_{ \{ P_x \} }) (f) = (-1)^n ~ \theta_{n+1} \llbracket P, f \rrbracket = (-1)^n \{ \! \! \{ \theta_1 (P), \theta_n (f) \} \! \! \} = \delta_\mathrm{mDA}(\theta_n (f) ) = (\delta_\mathrm{mDA} \circ \theta_n)(f).
\end{align*}
As a consequence, we get the desired result.
\end{proof}

In the following, we will consider a morphism from the cohomology of a matching relative Rota-Baxter algebra to the cohomology of the induced matching dendriform algebra. Note that the collection $\{\theta_n\}_{n=1}^{\infty}$ induces a collection of maps
\begin{align*}
C^n_{\mathrm{mrRBA}}(A,M,\{P_x\}_{x\in X}) &\rightarrow  C^{n}_{\mathrm{mDA}}(M,\{\prec_x,\succ_x\}_{x\in X}), \text{ for } n\geq 2\\
(\alpha,\beta,\gamma) &\mapsto \theta_{n-1}(\gamma).
\end{align*}
This collection also commutes with the corresponding differentials. Hence we obtain a morphism of cochain complexes from $\big\{ C^{\bullet}_{\mathrm{mrRBA}}(A,M,\{P_x\}_{x\in X}),\delta_{\mathrm{mrRBA}} \big\}$
 to $\big\{ C^{\bullet}_{\mathrm{mDA}} (M, \{ \prec_x, \succ_x \}_{x \in X} ),\delta_{\mathrm{mDA}} \big\}$. As a consequence, we obtain a morphism $H^{\bullet}_{\mathrm{mrRBA}}(A,M,\{P_x\}_{x\in X})\rightarrow H^{\bullet}_{\mathrm{mDA}}(M,\{\prec_x,\succ_x\}_{x\in X})$ between the corresponding cohomologies.

\subsection{Homotopy versions of matching dendriform algebras and matching relative Rota-Baxter algebras} In \cite{lod-val-book} Loday and Vallette considered a homotopy version of dendriform algebras (also called homotopy dendriform algebras or Dend$_\infty$-algebras) in their study of algebraic structures over operads. Recently, it has been explicitly studied in \cite{das-dend}. Like dendriform algebras are splitting of associative algebras, the concept of homotopy dendriform algebras are splitting of homotopy associative algebras or $A_\infty$-algebras. Our aim of the present subsection is to study the homotopy version of matching dendriform algebras.

Let $D$ be a graded vector space. Similar to the ungraded case, we can define the space $\mathcal{O}_D (n)$ associated to the graded vector space $D$. More precisely, for any $n \geq 1$, we have
\begin{align*}
  \mathcal{O}_D  (k) = \{ f = \oplus_{x_1, \ldots, x_k \in X} f_{x_1, \ldots, x_k} ~|~ & f_{x_1, \ldots, x_k} \in \mathrm{Hom}(\mathbb{K}[C_k] \otimes D^{\otimes k}, D) \text{ and } \\
  & f^{[r]}_{x_1, \ldots, x_k} \text{ doesn't depend on } x_r,~ \text{for} 1 \leq r \leq k  \}.
\end{align*}

\begin{defn}
A {\bf homotopy matching dendriform algebra} is a pair $(D, \{ \pi_k \}_{k \geq 1})$ consisting of a graded $\mathbb{K}$-module $D = \oplus_{i \in \mathbb{Z}} D^i$ equipped with a collection $\{ \pi_k \}_{k \geq 1}$ of elements $\pi_k \in \mathcal{O}_k (D)$ with $\mathrm{deg} \big(  (\pi_k)^{[r]}_{x_1, \ldots, x_k} \big) = k-2$ (for all $k \geq 1$, $[r] \in C_k$ and $x_1, \ldots, x_k \in X$) such that for all homogeneous elements $a_1, \ldots, a_n \in D$, elements $x_1, \ldots, x_n \in X$ and $[r] \in C_n$,
\begin{align}\label{mda-h}
\sum_{k+l=n+1} \sum_{i=1}^k (-1)^{i(l+1) + l (|a_1| + \cdots + |a_{i-1}|)}~ (\pi_k \circ_i \pi_l)^{[r]}_{x_1, \ldots, x_n} (a_1, \ldots, a_n) = 0,
\end{align}
where the terms  $(\pi_k \circ_i \pi_l)^{[r]}_{x_1, \ldots, x_n} (a_1, \ldots, a_n) $ are defined in (\ref{circ-dend-match}).
\end{defn}

\begin{remark} (i) A homotopy dendriform algebra \cite{lod-val-book,das-dend} can be regarded as a special case of a homotopy matching dendriform algebra in which the maps $(\pi_k)^{[r]}_{x_1, \ldots, x_k}$ (for all $k \geq 1$ and $[r] \in C_k$) doesn't depend on the elements $x_1, \ldots, x_k$.

(ii) Any matching dendriform algebra $(D, \{ \prec_x, \succ_x \}_{x \in X})$ can be thought of as a homotopy matching dendriform algebra $(D, \{ \pi_k \}_{k \geq 1})$ concentrated in degree $0$, where
\begin{align*}
\pi_k = 0 \text{ for } k \neq 2, \quad (\pi_2)^{[1]}_{x,y} (a,b) = a \prec_y b ~~ \text{ and } ~~ (\pi_2)^{[2]}_{x,y} (a,b) = a \succ_x b, \text{ for } x, y \in X \text{ and } a, b \in A.
\end{align*}
\end{remark}

In section \ref{sec3}, we have seen that a matching relative Rota-Baxter algebra induces a matching dendriform algebra structure (see the functor $\mathscr{F}$). In the following, we will generalize this result in the homotopy context. We first recall the notion of homotopy associative algebras and bimodules over them \cite{stasheff,keller}.

\begin{defn}
(i) An {\bf $A_\infty$-algebra} is a pair $(A, \{ \mu_k \}_{k \geq 1})$ consisting of a graded $\mathbb{K}$-module $A = \oplus_{i \in \mathbb{Z}} A^i$ equipped with a collection $\{ \mu_k : A^{\otimes k} \rightarrow A \}_{k \geq 1}$ of graded linear maps with $\mathrm{deg} (\mu_k) = k-2$ (for all $k \geq 1$) satisfying for all homogeneous elements $a_1, \ldots, a_{n} \in A$,
\begin{align}\label{a-inf-identities}
   \sum_{k+l = n+1} \sum_{i=1}^k (-1)^{i(l+1) + l (|a_1|+ \cdots + |a_{i-1}|)} \mu_k \big(  a_1, \ldots, a_{i-1}, \mu_l (a_i, \ldots, a_{i+l-1}), a_{i+l}, \ldots, a_n \big) = 0.
\end{align}

\medskip

(ii) Let $(A, \{ \mu_k \}_{k \geq 1})$ be an $A_\infty$-algebra. A {\bf representation} of this $A_\infty$-algebra consists of a pair $(M, \{ \eta_k \}_{k \geq 1})$ of a graded $\mathbb{K}$-module $M = \oplus_{i \in \mathbb{Z}} M^i$ and a collection $\{ \eta_k : \mathcal{A}^{k-1, 1} \rightarrow M \}_{k \geq 1}$ of graded linear maps with $\mathrm{deg} (\eta_k) = k-2$ (for all $k \geq 1$) satisfying the identities (\ref{a-inf-identities}) with exactly one of the variables $a_1, \ldots, a_n$ comes from $M$ and the corresponding linear operation $\mu_k$ or $\mu_l$ replaced by $\eta_k$ or $\eta_l$. Like the ungraded case, here $\mathcal{A}^{k-1,1}$ denotes the direct sum of all possible tensor powers of $A$ and $M$ in which $A$ appears $k-1$ times and $M$ appears exactly once.
\end{defn}

\begin{defn}
A {\bf homotopy matching relative Rota-Baxter algebra} is a triple $(A, M, \{ P_x \}_{x \in X})$ in which $(A, \{ \mu_k \}_{k \geq 1})$ is an $A_\infty$-algebra, $(M, \{ \eta_k \}_{k \geq 1})$ is a bimodule and $\{ P_x : M \rightarrow A \}_{x \in X}$ is a collection of degree $0$ linear maps satisfying
\begin{align}\label{homotopy-matching-rota-iden}
    \mu_k \big(  P_{x_1} (u_1), \ldots, P_{x_k} (u_k) \big) = \sum_{r=1}^k P_{x_r} \big(  \eta_k \big( P_{x_1}(u_1), \ldots, u_r, \ldots, P_{x_k} (u_k) \big) \big),
\end{align}
for all $u_1, \ldots, u_k \in M$, $x_1, \ldots, x_k \in X$ and $k \geq 1$.
\end{defn}

\begin{prop}
Let $(A, M, \{ P_x \}_{x \in X})$ be a homotopy matching relative Rota-Baxter algebra. Then $(M, \{ \pi_k \}_{k \geq 1})$ is a homotopy matching dendriform algebra, where
\begin{align*}
    (\pi_k)_{x_1, \ldots, x_k}^{[r]} (u_1, \ldots, u_k) = \eta_k \big( P_{x_1} (u_1), \ldots, u_r, \ldots, P_{x_k} (u_k)  \big),
\end{align*}
 for  $u_i \in M, x_i \in X ~(i=1, \ldots, k)$  and  $[r] \in C_k.$
\end{prop}

\begin{proof}
It follows from (\ref{homotopy-matching-rota-iden}) that
\begin{align*}
    \mu_k \big(  P_{x_1} (u_1), \ldots, P_{x_k} (u_k) \big) = \sum_{r=1}^k P_{x_r} \big( (\pi_k)^{[r]}_{x_1, \ldots, x_k} (u_1, \ldots, u_k)   \big).
\end{align*}
We observe that for any $k, l \geq 1$ with $k+l-1 = n$, and $1 \leq i \leq k$, we have either
\begin{align*}
    1 \leq r \leq i-1 \quad \text{ or } \quad i \leq r \leq i+l-1 \quad \text{ or } \quad i+l \leq r \leq k+l-1.
\end{align*}
If $1 \leq r \leq i-1$, then we have
\begin{align*}
    &(\pi_k \circ \pi_l)^{[r]}_{x_1, \ldots, x_n} (u_1, \ldots, u_n) \\
    &= \sum_{s=1}^l (\pi_k)^{[r]}_{x_1, \ldots, x_{i-1}, x_{i-1+s}, x_{i+l}, \ldots, x_n} \big( u_1, \ldots, u_{i-1}, (\pi_l)^{[s]}_{x_i, \ldots, x_{i+l-1}} (u_i, \ldots, u_{i+l-1}), u_{i+l}, \ldots, u_n  \big) \\
    %&~~+ (\pi_k)^{[r]}_{x_1, \ldots, x_{i-1}, x_{i+1}, x_{i+l}, \ldots, x_n} \big( u_1, \ldots, u_{i-1}, (\pi_l)^{[2]}_{x_i, \ldots, x_{i+l-1}} (u_i, \ldots, u_{i+l-1}), u_{i+l}, \ldots, u_n  \big) \\
    %&\quad\vdots \\
    %&~~~+ (\pi_k)^{[r]}_{x_1, \ldots, x_{i-1}, x_{i+l-1}, x_{i+l}, \ldots, x_n} \big( u_1, \ldots, u_{i-1}, (\pi_l)^{[l]}_{x_i, \ldots, x_{i+l-1}} (u_i, \ldots, u_{i+l-1}), u_{i+l}, \ldots, u_n  \big) \\
    &= \eta_k \big(  P_{x_1} (u_1), \ldots, u_r, \ldots, P_{x_{i-1}} (u_{i-1}), \sum_{s=1}^l P_{x_{i-1+s}} ( (\pi_l)^{[s]}_{x_i, \ldots, x_{i+l-1}} (u_i, \ldots, u_{i+l-1})), \ldots, P_{x_n} (u_n) \big) \\
    &= \eta_k \big(  P_{x_1} (u_1), \ldots, u_r, \ldots, P_{x_{i-1}} (u_{i-1}), \mu_l \big(  P_{x_i} (u_i), \ldots, P_{i+l-1} (u_{i+l-1}) \big), \ldots, P_{x_n} (u_n)  \big).
\end{align*}
Similarly, if $i \leq r \leq i+l-1$, we get the term as
\begin{align*}
 &(\pi_k \circ \pi_l)^{[r]}_{x_1, \ldots, x_n} (u_1, \ldots, u_n) \\
 &= (\pi_k)^{[i]}_{x_1, \ldots, x_{i-1}, x_r, x_{i+l}, \ldots, x_n} \big( u_1, \ldots, u_{i-1}, (\pi_l)^{[r-i+1]}_{x_1, \ldots, x_{i+l-1}} (u_i, \ldots, u_{i+l-1}), \ldots, P_{x_n} (u_n)    \big)  \\
 &= \eta_k \big( P_{x_1} (u_1), \ldots, P_{x_{i-1}} (u_{i-1}), (\pi_l)_{x_i, \ldots, x_{i+l-1}}^{[r-i+1]} (u_i, \ldots, u_{i+l-1}), \ldots, P_{x_n} (u_n)   \big) \\
 &= \eta_k \big(  P_{x_1}(u_1), \ldots,  P_{x_{i-1}} (u_{i-1}), \eta_l \big(  P_{x_i} (u_i), \ldots, u_r, \ldots, P_{x_{i+l-1}} (u_{i+l-1}) \big), \ldots, P_{x_n} (u_n)  \big).
 \end{align*}
 Finally, if $i +l \leq r \leq k+l-1$, we get
 \begin{align*}
 &(\pi_k \circ \pi_l)^{[r]}_{x_1, \ldots, x_n} (u_1, \ldots, u_n) \\
 &= \sum_{s=1}^l (\pi_k)^{[r-l+1]}_{x_1, \ldots, x_{i-1}, x_{i-1+s}, x_{i+l}, \ldots, x_n} \big( u_1, \ldots, u_{i-1}, (\pi_l)^{[s]}_{x_i, \ldots, x_{i+l-1}} (u_i, \ldots, u_{i+l-1}), u_{i+l}, \ldots, u_n  \big) \\
   % &~~~+ (\pi_k)^{[r-l+1]}_{x_1, \ldots, x_{i-1}, x_{i+1}, x_{i+l}, \ldots, x_n} \big( u_1, \ldots, u_{i-1}, (\pi_l)^{[2]}_{x_i, \ldots, x_{i+l-1}} (u_i, \ldots, u_{i+l-1}), u_{i+l}, \ldots, u_n  \big) \\
   % &\quad \vdots \\
  %  &~~~+ (\pi_k)^{[r-l+1]}_{x_1, \ldots, x_{i-1}, x_{i+l-1}, x_{i+l}, \ldots, x_n} \big( u_1, \ldots, u_{i-1}, (\pi_l)^{[l]}_{x_i, \ldots, x_{i+l-1}} (u_i, \ldots, u_{i+l-1}), u_{i+l}, \ldots, u_n  \big) \\
    &= \eta_k \big(  P_{x_1} (u_1), \ldots, P_{x_{i-1}} (u_{i-1}), \sum_{s=1}^l P_{x_{i-1+s}} ( (\pi_l)^{[s]}_{x_i, \ldots, x_{i+l-1}} (u_i, \ldots, u_{i+l-1})), \ldots, u_r \ldots, P_{x_n} (u_n) \big) \\
    &= \eta_k \big(  P_{x_1} (u_1),  \ldots, P_{x_{i-1}} (u_{i-1}), \mu_l \big(  P_{x_i} (u_i), \ldots, P_{i+l-1} (u_{i+l-1}) \big),  P_{x_{i+l}} (u_{i+l}), \ldots, u_r, \ldots, P_{x_n} (u_n)  \big).
 \end{align*}
 Therefore, the identities (\ref{mda-h}) follow as $(M, \{ \eta_k \}_{k \geq 1})$ is a representation of the $A_\infty$-algebra $(A, \{ \mu_k \}_{k \geq 1})$. This completes the proof.
\end{proof}

In section \ref{sec3}, we also shown that a matching dendriform algebra gives rise to a matching relative Rota-Baxter algebra such that the induced matching dendriform algebra coincides with the given one (see the functor $\mathscr{G}$). This result also has a generalization in the homotopy context.

\begin{thm}
Let $(D, \{ \pi_k \}_{k \geq 1})$ be a homotopy matching dendriform algebra. Then the followings are hold.

(i) The pair $(D \otimes \mathbb{K}[X], \{ \mu_k \}_{k \geq 1})$ is an $A_\infty$-algebra, where
\begin{align}\label{define-a-inf-mul}
    \mu_k \big(  a_1 \otimes x_1, \ldots, a_k \otimes x_k \big) := \sum_{r=1}^k \big(  (\pi_k)^{[r]}_{x_1, \ldots, x_k} (a_1, \ldots, a_k)   \big) \otimes x_r, \text{ for } a_i \otimes x_i \in D \otimes \mathbb{K}[X].
\end{align}

(ii) The graded vector space $D$ can be given a bimodule structure over the $A_\infty$-algebra $(D \otimes \mathbb{K}[X], \{ \mu_k \}_{k \geq 1})$ with the action maps $\{ \eta_k \}_{k \geq 1}$ are given by
\begin{align}\label{define-action-mul}
    \eta_k \big( a_1 \otimes x_1, \ldots, a_r, \ldots, a_k \otimes x_k   \big) := (\pi_k)^{[r]}_{x_1, \ldots, x_k} (a_1, \ldots, a_k),
\end{align}
for $a_i \otimes x_i \in D \otimes \mathbb{K}[X]$ $(i=1, \ldots, r-1, r+1, \ldots, k)$ and $a_r \in D$. On the right hand side of the above defining identity, we consider $x_r$ to be any element of $X$.

(iii) The triple $\big( (D \otimes \mathbb{K}[X], \{ \mu_k \}_{k \geq 1}), (D, \{ \eta_k \}_{k \geq 1}) , \{ P_x \}_{x \in X}  \big)$ is a homotopy matching relative Rota-Baxter algebra, where $\{ P_x : D \rightarrow D \otimes \mathbb{K}[X] \}_{x \in X}$ is the collection of degree $0$ linear maps given by
\begin{align*}
    P_x (a) := a \otimes x, \text{ for } a \in D \text{ and } x \in X.
\end{align*}
Moreover, the induced homotopy matching dendriform algebra structure on $D$ coincides with the given one.
\end{thm}

\begin{proof}
(i) If $(D, \{ \pi_k \}_{k \geq 1})$ is a homotopy matching dendriform algebra, then it is not hard to see that $( D \otimes \mathbb{K}[X] , \{ \overline{\pi}_k  \}_{k \geq 1})$ is a homotopy dendriform algebra, where
\begin{align*}
( \overline{\pi}_k)^{[r]} (a_1 \otimes x_1, \ldots, a_k \otimes x_k) := (\pi_k)^{[r]}_{x_1, \ldots, x_k} (a_1, \ldots, a_k) \otimes x_r, \text{ for } [r] \in C_k.
\end{align*}
As a consequence,  $( D \otimes \mathbb{K}[X] , \{ \mu_k  \}_{k \geq 1})$ is an $A_\infty$-algebra, where $\mu_k = \sum_{r=1}^k ( \overline{\pi}_k)^{[r]}$. Hence the result follows.

\medskip

(ii) Let $a_i \otimes x_i \in D \otimes \mathbb{K}[X]$ (for $i=1, \ldots, r-1, r+1, \ldots, n$) and $a_r \in D$. Let $k, l \geq 1$ with $k+l-1=n$ and $1 \leq i \leq k$. Then for any $[r] \in C_n$, we have either
\begin{align*}
1 \leq r \leq i-1 \quad \text{ or } \quad i \leq r \leq i+l-1 \quad \text{ or } \quad i+l \leq r \leq k+l-1.
\end{align*}
In any case, it is easy to see that the corresponding bimodule condition is equivalent to the homotopy matching dendriform algebra condition (\ref{mda-h}). Hence the proof follows.

\medskip

(iii) For any $a_1, \ldots, a_k \in D$ and $x_1, \ldots, x_k \in X$, we have
\begin{align*}
    \mu_k \big(  P_{x_1} (a_1), \ldots, P_{x_k} (a_k) \big) =~& \mu_k \big(  a_1 \otimes x_1, \ldots, a_k \otimes x_k \big) \\
    =~& \sum_{r=1}^k P_{x_r} \big(  (x_k)^{[r]}_{x_1, \ldots, x_k} (a_1, \ldots, a_k)  \big)  \quad (\text{by } (\ref{define-a-inf-mul}))\\
    =~& \sum_{r=1}^k P_{x_r} \big( \eta_k (P_{x_1} (a_1), \ldots, a_r, \ldots, P_{x_k} (a_k)  \big) \quad (\text{by } (\ref{define-action-mul})).
\end{align*}
This shows that the triple $(D \otimes \mathbb{K}[X] , D, \{ P_x \}_{x \in X})$ is a homotopy matching relative Rota-Baxter algebra. If $(D, \{ \widetilde{\pi}_k \}_{k \geq 1})$ is the induced homotopy matching dendriform algebra structure on $D$, then we have
\begin{align*}
    (\widetilde{\pi}_k)^{[r]}_{x_1, \ldots, x_k } (a_1, \ldots, a_k) = \eta_k \big( P_{x_1}(a_1), \ldots, a_r, \ldots, P_{x_k} (a_k)  \big) = (\pi_k)^{[r]}_{x_1, \ldots, x_k} (a_1, \ldots, a_k) \quad (\text{by } (\ref{define-action-mul})).
\end{align*}
This completes the proof.
\end{proof}

\medskip

\medskip

\noindent {\bf Acknowledgements.} Ramkrishna Mandal would like to thank CSIR, Government of India for funding the PhD fellowship. Both the authors thank Department of Mathematics, IIT kharagpur for providing the beautiful academic atmosphere where the research has been carried out.
%The author would like to thank Indian Institute of Technology (IIT) Kharagpur for providing the beautiful academic environment where the research has been carried out.

\medskip

\noindent {\bf Data Availability Statement.} Data sharing is not applicable to this article as no new data were created or analyzed in this study.

\end{document}